\RequirePackage{fix-cm}
\documentclass{svjour3}                     % onecolumn (standard format)
\smartqed  % flush right qed marks, e.g. at end of proof
\usepackage{graphicx}
\usepackage{amsmath}
\usepackage{amsfonts}
\usepackage{amssymb}
\begin{document}

\title{Optimal control of systems governed by Dirichlet fractional Laplacian in the minimax framework
%{Grants or other notes
}


\titlerunning{Control problems governed by spectral Dirichlet fractional Laplacian}        % if too long for running head

\author{Dorota Bors}


\institute{Dorota Bors \at
              Faculty of Mathematics and Computer Science, University of Lodz\\
             S. Banacha 22, 90-238 Lodz, Poland \\
              \email{bors@math.uni.lodz.pl}
}


\maketitle

\begin{abstract}
We consider optimal control problem governed by a class of partial differential equations with the spectral Dirichlet fractional Laplacian. Some sufficient condition for the existence of optimal processes are proved. The proof of the main result relies on variational structure of the problem. To show that partial differential equations with the Dirichlet fractional Laplacian have a weak solution we employ the renowned Ky Fan Theorem.

%\PACS{PACS code1 \and PACS code2 \and more}
\end{abstract}

\section{Introduction}

Let $\Omega\subset\mathbb{R}^{n}$ for $n\geq3$ be a bounded domain with a
Lipschitz boundary$.$ We consider a boundary value problem for
nonlinear nonlocal vector equation
\begin{equation}
(-\Delta)^{\alpha/2}\psi\left(  x\right)  +f\left(  x,\psi\left(  x\right)
,w\left(  x\right)  \right)  =0\text{ in }\Omega\text{,} \label{0.1}%
\end{equation}%
\begin{equation}
\psi\left(  x\right)  =0\text{ on }\partial\Omega\label{0.2}%
\end{equation}
where a vector function $\psi$ belongs to some fractional Sobolev space
$H_{0}^{\alpha/2},$ a control $w$ belongs to $L^{p}$ and $\alpha\in(1,2)$.
The problems involving different notions of the fractional Laplacian attracted
in the recent years a lot of attention motivated by the problems in finances
\cite{App}, mechanics \cite{BerSag,BogBycKul} hydrodynamics
\cite{BonVaz,CafSalSil,CafVas,Vaz2}, elastostatics \cite{BerSag} or
probability \cite{App,BogBycKul,CheSon}. It should be moreover noted that at
least two notions of fractional Laplace operator coexist: the first the
Dirichlet fractional Laplacian defined by the spectral properties of the
Dirichlet Laplace operator, see \cite{BarColPab,CabTan} and the second one
defined via the singular integral or the infinitesimal generator of the
L\'{e}vy semigroup, for a list of relevant references, see
\cite{App,BanKul,BogBycKul,CheSon,Val}. In this paper we use the
Dirichlet fractional Laplacian set in the spectral framework. The problems
governed by the Dirichlet fractional Laplacian can be seen as a natural
extensions of the problems discussed in \cite{Bor1,Bor2,WalLedSch}
involving the standard Laplace operator. Specifically, we focus our attention
on the  of the solutions on the functional parameters and
then on the existence of the optimal solutions minimizing some cost
functional. For related results concerning optimal solution we refer the
interested readers, for example, to papers \cite{BerSag,Bor2,DalGer,WalLedSch}. The
framework requires the minimax geometry (cf. \cite{Rab,Wil}) for
concave-convex functionals of action allowing by Ky Fan Theorem the existence
of saddle point solutions. For related results involving some notions of the
fractional Laplacian, see, among others, papers
\cite{BraColPab,SerVal2} with the minimax geometry setting.

To be more specific, consider a problem with boundary data $u=v =0$ on $\partial\Omega$
\begin{equation}
\left\{
\begin{array}
[c]{l}%
-(-\Delta)^{\alpha/2}u\left(  x\right)  +G_{u}\left(  x,u\left(  x\right)
,v\left(  x\right)  ,w\left(  x\right)  \right)  =0\\
(-\Delta)^{\alpha/2}v\left(  x\right)  +G_{v}\left(  x,u\left(  x\right)
,v\left(  x\right)  ,w\left(  x\right)  \right)  =0
\end{array}
\right.  \label{0.3}%
\end{equation}
Clearly, the above problem is a particular case of
$(\ref{0.1})-(\ref{0.2})$ with $\psi=\left(  -u,v\right)  $ and $f=\left(
G_{u},G_{v}\right)  .$ We prove in Section \ref{existence_section} that
control problem $(\ref{0.3})$ possesses at least one weak solution
for any control $w$. The results concerning the continuous dependence of weak
solution on controls are discussed in Section \ref{dependence_section}.
Without going into details, for a given control $w_{k},$ denote by $\left(
u_{k},v_{k}\right)  $ a weak solution of problem $(\ref{0.3})$,
then if the sequence $\left\{w_{k}\right\}$ tends to
$w_{0}$ in appropriate topology of $L^{p},$ then the sequence $\left\{
\left(  u_{k},v_{k}\right)  \right\} $ tends to $\left(
u_{0},v_{0}\right)  $ in the strong topology of $H_{0}^{\alpha/2}\times
H_{0}^{\alpha/2}.$ In other words, we have proved that boundary value problem
$(\ref{0.3})$ is well-posed, i.e. the solution exists and it
continuously depends on controls. Section \ref{optimal_section} is devoted to
the investigation of optimal control problem. The proof of the existence of
the optimal solution, which is the main result of the paper, relies on the
continuous dependence results from Section \ref{dependence_section}. Finally,
some examples are presented.

%-----------------------------------------------

\section{Statement of the problem\label{statement_section}}

Throughout the paper, we shall assume that $\Omega\subset\mathbb{R}^{n}$ with
$n>2$ is a bounded domain with a Lipschitz boundary. 
Moreover, we shall use spectral properties of the
fractional Laplacian in the case of bounded domain $\Omega$ with smooth
boundary. The powers $\left(  -\Delta\right)  ^{\alpha/2}$ of the positive
Laplace operator $\left(  -\Delta\right)  ,$ in a bounded domain with zero
Dirichlet boundary data are defined through the spectral decomposition using
the powers of the eigenvalues of the original operator. Let $(z_{k},\rho_{k})$
for $k\in\mathbb{N}$ be the system of the eigenfunctions and eigenvalues of
the Laplace operator $\left(  -\Delta\right)  $ on $\Omega$ with the
homogeneous Dirichlet condition on $\partial\Omega$. Then $(z_{k},\rho
_{k}^{\alpha/2})$ for $k\in\mathbb{N}$ is the system of the eigenfunctions and
eigenvalues of the fractional Laplacian $\left(  -\Delta\right)  ^{\alpha/2}$
on $\Omega,$ also with the homogeneous boundary Dirichlet condition$.$ By
$H_{0}^{\alpha/2}\left(  \Omega\right)  ,$ we can denote the space of
functions $z=z\left(  x\right)  $ defined on a bounded, smooth domain
$\Omega\subset\mathbb{R}^{n},\;n\geq3,$ such that $z=\sum_{k=1}^{\infty}%
a_{k}z_{k}$ and $\sum_{k=1}^{\infty}a_{k}^{2}\rho_{k}^{\alpha/2}<\infty$, with
the norm defined by the formula
\[
\left\Vert z\right\Vert _{H_{0}^{\alpha/2}\left(  \Omega\right)  }^{2}%
=\sum_{k=1}^{\infty}a_{k}^{2}\rho_{k}^{\alpha/2}=\left\Vert (-\Delta
)^{\alpha/4}z\right\Vert _{L^{2}\left(  \Omega\right)  }^{2},
\]
cf. \cite[Prop. 4.4]{NezPalVal}. The Dirichlet fractional Laplacian acts
on $z=\sum_{k=1}^{\infty}a_{k}z_{k}$
%\in H_{0}^{\alpha/2}\left(  \Omega\right)
as
\[
(-\Delta)^{\alpha/2}z=\sum_{k=1}^{\infty}a_{k}\rho_{k}^{\alpha/2}z_{k}.
\]
There exists also a different notion of the fractional Laplacian, defined via
singular integral on the whole of $\mathbb{R}^{n}$ 
%as
%\[
%(-\Delta)^{\alpha/2}z\left(  x\right)  =-\tfrac{1}{2}\int_{\mathbb{R}^{n}%
%}\tfrac{z\left(  x+y\right)  +z\left(  x-y\right)  -2z\left(  x\right)
%}{\left\vert x-y\right\vert ^{n+\alpha}}dy
%\]
%for all $x\in\mathbb{R}^{n}$ 
which can be restricted to the functions with
some values on $\Omega$ and zero value outside the set $\Omega.$ It should be
underlined, however, that it leads to nonequivalent definition and therefore
is often referred to as the restricted fractional Laplacian as in
\cite{BonVaz,RosSer1} not to be confused with the spectral Dirichlet
fractional Laplacian used in this paper. For the differences between two
notions of the fractional Laplacian one can see, for example,
\cite{BanKul,BogBycKul,CheSon,SerVal4}, where the spectral analysis of both
the operators were carried over.

It is worth reminding the reader that for a bounded domain with a Lipschitz
boundary, the fractional Sobolev space $H^{\alpha/2}\left(  \Omega\right)  $
is compactly embedded into $L^{s}\left(  \Omega\right)  $ for $s\in\left[
1,2_{\alpha}^{\ast}\right)  $ where $2_{\alpha}^{\ast}=2n/\left(
n-\alpha\right)  $ for $n>2$ and the inequality%
\[
\left\Vert
z\right\Vert _{L^{s}\left(  \Omega\right)  }\leq C\left\Vert z\right\Vert _{H^{\alpha/2}\left(  \Omega\right)  }%
\]
holds, cf. \cite[Corollary 7.2]{NezPalVal} and \cite{BarColPab}. Recall the fractional Poincar\'{e} inequality
\begin{equation}
\rho_{1}^{\alpha/2}\int\nolimits_{\Omega}\left\vert u\left(  x\right)
\right\vert ^{2}dx\leq\int\nolimits_{\Omega}\left\vert (-\Delta)^{\alpha
/4}u\left(  x\right)  \right\vert ^{2}dx. \label{fPoincare}%
\end{equation}
Note that $\rho_{1}$ is the principal eigenvalue of the Laplacian and  $(-\Delta)^{\alpha/4}u_{1}=\rho_{1}^{\alpha/4}u_{1}.$ Moreover, 
$\left\Vert u\right\Vert _{H_{0}^{\alpha/2}\left(  \Omega\right)  }^{2}=\int
\nolimits_{\Omega}\left\vert (-\Delta)^{\alpha/4}u\left(  x\right)
\right\vert ^{2}dx$ is weakly lower semicontinuous, convex and coercive as the
norm in the reflexive space, for details see \cite{AutPuc,Gre}.

In this paper we consider systems of nonlinear fractional differential
equations%
\begin{equation}
\left\{
\begin{array}
[c]{l}%
-(-\Delta)^{\alpha/2}u\left(  x\right)  +G_{u}\left(  x,u\left(  x\right)
,v\left(  x\right)  ,w\left(  x\right)  \right)  =0\\
(-\Delta)^{\alpha/2}v\left(  x\right)  +G_{v}\left(  x,u\left(  x\right)
,v\left(  x\right)  ,w\left(  x\right)  \right)  =0\\
u\left(  x\right)  =0,\text{ }v\left(  x\right)  =0
\end{array}
\right.
\begin{array}
[c]{l}%
\text{in }\Omega\\
\text{in }\Omega\\
\text{on }\partial\Omega
\end{array}
\label{1.1}%
\end{equation}
where $u\in H_{0}^{\alpha/2}\left(  \Omega\right)  $, $v\in H_{0}^{\alpha
/2}\left(  \Omega\right)  ,$ $G$ is a scalar function on
$\Omega\times\mathbb{R}^{2+m}$ and $w\in\mathcal{W}=\left\{  w\in L^{p}\left(  \Omega,\mathbb{R}^{m}\right)  :w\left(
x\right)  \in M\text{ for a.e. }x\in\Omega\right\},$
where $M\subset\mathbb{R}^{m}$ is convex and bounded. $\mathcal{W}$
will be referred to as a set of distributed parameters or controls.

We shall investigate the question of the continuous dependence on control
$w\in\mathcal{W}$ of weak solutions of problem $(\ref{1.1})$ in the space
$\mathbb{H}_{0}^{\alpha/2}=H_{0}^{\alpha/2}\left(  \Omega\right)  \times
H_{0}^{\alpha/2}\left(  \Omega\right)  .$ We replace this question, under some
assumption about the function $G=G\left(  x,u,v,w\right)  ,$ with the question
of the continuous dependence on controls of saddle points of the functional of
action $F_{w}\left(  u,v\right)$ for problem $(\ref{1.1})$ of the form
\begin{equation}
\int\nolimits_{\Omega}\left(  \tfrac{1}{2}\left\vert
(-\Delta)^{\alpha/4}v\left(  x\right)  \right\vert ^{2}-\tfrac{1}{2}\left\vert
(-\Delta)^{\alpha/4}u\left(  x\right)  \right\vert ^{2}+G\left(  x,u\left(
x\right)  ,v\left(  x\right)  ,w\left(  x\right)  \right)  \right)  dx,
\label{1.2}%
\end{equation}
defined on the space $\mathbb{H}_{0}^{\alpha/2}$ with the norm
$
\left\Vert \left(  u,v\right)  \right\Vert _{\mathbb{H}_{0}^{\alpha/2}}%
^{2}=\left\Vert u\right\Vert _{H_{0}^{\alpha/2}\left(  \Omega\right)  }%
^{2}+\left\Vert z\right\Vert _{H_{0}^{\alpha/2}\left(  \Omega\right)  }^{2}.
$
Let us recall that a pair $\left(  u_{0},v_{0}\right)  \in\mathbb{H}%
_{0}^{\alpha/2}$ is a saddle point of a functional $F_{w}$ if
%$\mathbb{H}_{0}^{\alpha/2}\rightarrow\mathbb{R}$ if%
\[
F_{w}\left(  u,v_{0}\right)  \leq F_{w}\left(  u_{0},v_{0}\right)  \leq
F_{w}\left(  u_{0},v\right)
\]
for any $u\in H_{0}^{\alpha/2}\left(  \Omega\right)  $ and $v\in H_{0}%
^{\alpha/2}\left(  \Omega\right)  $ which is equivalent to
\[
\sup_{u}\inf_{v}F_{w}\left(  u,v\right)  =\inf_{v}\sup_{u}F_{w}\left(
u,v\right)  =F_{w}\left(  u_{0},v_{0}\right)
\]
provided that $\sup_{u}\inf_{v}F_{w}\left(  u,v\right)  $ and $\inf_{v}%
\sup_{u}F_{w}\left(  u,v\right)  $ are finite and attainable. Moreover, a pair
$\left(  u,v\right)  \in\mathbb{H}_{0}^{\alpha/2}$ is the the weak solution of
problem $(\ref{1.1})$ if, for any $\left(  g,h\right)  \in\mathbb{H}%
_{0}^{\alpha/2},$ the following equalities, compare with \cite[Definition 2.1]{BarColPab}, hold
\[
\left\{
\begin{array}
[c]{r}%
-\int\limits_{\Omega}(-\Delta)^{\alpha/4}u\left(  x\right)  (-\Delta
)^{\alpha/4}g\left(  x\right)  dx+\int\limits_{\Omega}G_{u}\left(  x,u\left(
x\right)  ,v\left(  x\right)  ,w\left(  x\right)  \right)  g\left(  x\right)
dx=0\text{ in }\Omega,\\
\int\limits_{\Omega}(-\Delta)^{\alpha/2}v\left(  x\right)  (-\Delta
)^{\alpha/4}h\left(  x\right)  dx+\int\limits_{\Omega}G_{v}\left(  x,u\left(
x\right)  ,v\left(  x\right)  ,w\left(  x\right)  \right)  h\left(  x\right)
dx=0\text{ in }\Omega.
\end{array}
\right.
\]

Let us make the following assumptions:

\begin{enumerate}
\item[(A1)] $G,G_{u},G_{v}$ are Carath\'{e}odory functions, i.e. they are
measurable with respect to $x$ for any $\left(  u,v,w\right)  \in
\mathbb{R}^{2+m}$ and continuous w. r. t. $\left(  u,v,w\right)  $ for
a.e. $x\in\Omega$;

\item[(A2)] for $p=\infty,$ there exists $c>0$ such that for $z\in\{u,v\}$
\begin{align*}
\left\vert G\left(  x,u,v,w\right)  \right\vert  &  \leq c\left(  1+\left\vert
u\right\vert ^{s}+\left\vert v\right\vert ^{s}\right)  ,\\
\left\vert G_{z}\left(  x,u,v,w\right)  \right\vert  &  \leq c\left(
1+\left\vert u\right\vert ^{s-1}+\left\vert v\right\vert ^{s-1}\right)  ,
%\\
%\left\vert G_{v}\left(  x,u,v,w\right)  \right\vert  &  \leq c\left(
%1+\left\vert u\right\vert ^{s-1}+\left\vert v\right\vert ^{s-1}\right)  ,
\end{align*}
where $s\in\left(  1,2_{\alpha}^{\ast}\right)  $ for $n\geq3$ and $2_{\alpha
}^{\ast}=\frac{2n}{n-\alpha}$, $x\in\Omega$ a.e., $u\in\mathbb{R},$
$v\in\mathbb{R}$ and $w\in M;$ if $p\in\lbrack1,\infty),$ there exists
$c>0$ such that for any $z\in\{u,v\}$
\begin{align*}
\left\vert G\left(  x,u,v,w\right)  \right\vert  &  \leq c\left(  1+\left\vert
u\right\vert ^{s}+\left\vert v\right\vert ^{s}+\left\vert w\right\vert
^{p}\right)  ,\\
\left\vert G_{z}\left(  x,u,v,w\right)  \right\vert  &  \leq c\left(
1+\left\vert u\right\vert ^{s-1}+\left\vert v\right\vert ^{s-1}+\left\vert
w\right\vert ^{p-\frac{p}{s}}\right)  ,
%\\
%\left\vert G_{v}\left(  x,u,v,w\right)  \right\vert  &  \leq c\left(
%1+\left\vert u\right\vert ^{s-1}+\left\vert v\right\vert ^{s-1}+\left\vert
%w\right\vert ^{p-\frac{p}{s}}\right)  ,
\end{align*}
where $s\in\left(  1,2_{\alpha}^{\ast}\right)  $ for $n\geq3$ and a.e.
$x\in\Omega$, $u\in\mathbb{R}$, $v\in\mathbb{R}$ and $w\in\mathbb{R}^{m};$

\item[(A3)] for any $u\in H_{0}^{\alpha/2}\left(  \Omega\right)  $, there
exist $b\in\mathbb{R}$, $\beta_{1}\in
L^{2}\left(  \Omega\right)  $, $\gamma_{1}\in L^{1}\left(  \Omega\right)  $,
such that
\[
G\left(  x,u\left(  x\right)  ,v,w\right)  \geq-b\left\vert v\right\vert
^{2}-\beta_{1}\left(  x\right)  v-\gamma_{1}\left(  x\right)
\]
for any $v\in\mathbb{R}$, $w\in M$ and a.e. $x\in\Omega$, where $\rho
_{1}^{\alpha/2}>2b$ and $\rho_{1}$ is the principal eigenvalue of the Laplace
operator with the zero Dirichlet boundary values;

\item[(A4)] for any $v\in H_{0}^{\alpha/2}\left(  \Omega\right)  $, there
exist $B\in\mathbb{R}$, $\beta_{2}\in
L^{2}\left(  \Omega\right)  $, $\gamma_{2}\in L^{1}\left(  \Omega\right)  $,
such that
\[
G\left(  x,u,v\left(  x\right)  ,w\right)  \leq B\left\vert u\right\vert
^{2}+\beta_{2}\left(  x\right)  u+\gamma_{2}\left(  x\right)
\]
for any $u\in\mathbb{R}$, $w\in M$ and $x\in\Omega$ a.e., where $\rho
_{1}^{\alpha/2}>2B$ and $\rho_{1}$ is the principal eigenvalue of the Laplace
operator with the zero Dirichlet boundary values$;$

\item[(A5)] for any $w\in\mathcal{W},$ the functional $F_{w}$ is concave with
respect to $u$ for any $v\in H_{0}^{\alpha/2}\left(  \Omega\right)  $ and
convex with respect to $v$ for any $u\in H_{0}^{\alpha/2}\left(
\Omega\right)  $; shortly, for any $w\in\mathcal{W},$ the functional $F_{w}$
is concave-convex, where $F_{w}$ is defined in $(\ref{1.2})$.
\end{enumerate}

%\begin{remark}
%Under assumptions $(A1)$ and $(A2)$, for any $w\in\mathcal{W},$ functional
%$F_{w}$ defined in $(\ref{1.2})$ is well-defined and of $C^{1}$-class with
%respect to $u$ and $v$, cf. \cite[Theorems C.1 and C.2]{Str}.
%\end{remark}

%\begin{remark}
%Directly from assumptions $(A1)$, $\left(  A2\right)  $, $(A3)$, $(A4)$ and
%\cite[Theorem 1.6]{Str}, it follows that, for any $w\in\mathcal{W}$, the
%functional $F_{w}$ is weakly lower semicontinuous with respect to $v$ for any
%$u\in H_{0}^{^{\alpha/2}}\left(  \Omega\right)  $ and weakly upper
%semicontinuous with respect to $u$ for any $v\in H_{0}^{^{\alpha/2}}\left(
%\Omega\right)  .$
%\end{remark}
%-----------------------------------------
\section{Existence of saddle points\label{existence_section}}

In this section we shall focus our attention on study of the variational
formulation of problem associated with fractional differential system $\left(
\ref{1.1}\right)  $. We shall prove that for any $w\in\mathcal{W},$ there
exists a saddle point of the function of action defined in $(\ref{1.2}).$
Moreover, we shall demonstrate that the set of all saddle points is bounded.
In doing this we will also benefit from having the following notation. For any
$w\in\mathcal{W}$ denote by $S_{w}$ the set all saddle point of $F_{w}$, i.e.
\[
S_{w}=\left\{  \left(  u_{w},v_{w}\right)  \in\mathbb{H}_{0}^{\alpha/2}%
:F_{w}\left(  u,v_{w}\right)  \leq F_{w}\left(  u_{w},v_{w}\right)  \leq
F_{w}\left(  u_{w},v\right)  \right\}  .
\]
To prove that $F_{w}$ possesses the saddle point we shall apply the 
Ky Fan's Theorem \cite[Theorem 5.2.2]{Nir}.
%\begin{theorem}
%[{\cite[Theorem 5.2.2]{Nir}}]Let $X$, $Y$ be any linear topological spaces,
%$A\subset X$, $B\subset Y$ be some convex sets. Let $F:A\times B\rightarrow
%\mathbb{R}$ be any function satisfying conditions:\newline(a) for any $x\in A$
%the function $F\left(  x,\cdot\right)  $ is convex and lower
%semicontinous,\newline(b) for any $y\in B$ the function $F\left(
%\cdot,y\right)  $ is concave and upper semicontinous,\newline(c) there exist
%$y_{0}\in B$ and $\lambda$ such that \textrm{inf}$_{y\in B}\sup_{x\in
%A}F\left(  x,y\right)  >\lambda$ and the set $\left\{  x\in A:F\left(
%x,y_{0}\right)  \geq\lambda\right\}  $ is compact,\newline then $\sup_{x\in
%A}$\textrm{inf}$_{y\in B}F\left(  x,y\right)  =$\textrm{inf}$_{y\in B}%
%\sup_{x\in A}F\left(  x,y\right)  .$
%\end{theorem}
Now we provide the statement of the theorem on the following properties of the
set of saddle points: nonemptiness and boundedness.

\begin{theorem}
[On the existence of saddle points]\label{existence}If conditions $(A1)-(A5)$
are satisfied, then for any $w\in\mathcal{W}$, there exists at least one
saddle point $\left(  u_{w},v_{w}\right)  \in\mathbb{H}_{0}^{\alpha/2}$ for
the functional $F_{w}$ defined in $(\ref{1.2})$, and moreover there are some
balls $B_{1}\left(  0,r_{1}\right)  \subset H_{0}^{\alpha/2}\left(
\Omega\right)  $ and $B_{2}\left(  0,r_{2}\right)  \subset H_{0}^{\alpha
/2}\left(  \Omega\right)  $ such that, for all $w\in\mathcal{W},$
$S_{w}\subset B_{1}\left(  0,r_{1}\right)  \times B_{2}\left(  0,r_{2}\right)
\subset\mathbb{H}_{0}^{\alpha/2}.$\newline If the functional $F_{w}$ is
additionally assumed to be strictly concave - strictly convex, then the saddle
point is unique.
\end{theorem}

\begin{proof}
Let $w\in\mathcal{W}$ be fixed$.$ First note that the functional $F_{w}\left(
u,\cdot\right)  $ is coercive for any $u\in H_{0}^{^{\alpha/2}}\left(
\Omega\right)  .$ From assumption $(A3)$, for any $u\in H_{0}^{^{\alpha/2}%
}\left(  \Omega\right)  $, there exist a constant $b$ and functions $\beta
_{1}\in L^{2}\left(  \Omega\right)  $, $\gamma_{1}\in L^{1}\left(
\Omega\right)  $ such that
\[
F_{w}\left(  u,v\right)  \geq\int\nolimits_{\Omega}\left(  \tfrac{1}%
{2}\left\vert (-\Delta)^{\alpha/4}v\left(  x\right)  \right\vert
^{2}-b\left\vert v\left(  x\right)  \right\vert ^{2}-\beta_{1}\left(
x\right)  v\left(  x\right)  -\gamma_{1}\left(  x\right)  \right)  dx.
\]
The application of the fractional Poincar\'{e} inequality $\left(
\ref{fPoincare}\right)  $ and the Schwartz inequality lead to the following
estimate:
\[
F_{w}\left(  u,v\right)  \geq\left(  \tfrac{1}{2}-b\rho_{1}^{-\alpha
/2}\right)  \left\Vert v\right\Vert _{H_{0}^{^{\alpha/2}}}^{2}-C_{1}\left\Vert
v\right\Vert _{H_{0}^{^{\alpha/2}}}-C_{2}%
\]
where $C_{1}$, $C_{2}$ are some nonnegative constants. Since $\frac{1}%
{2}-b\rho_{1}^{-\alpha/2}>0$, the functional $F_{w}\left(  u,\cdot\right)  $
is coercive. As a result, for any $u\in H_{0}^{^{\alpha/2}}\left(
\Omega\right)  ,$ the functional $F_{w}\left(  u,\cdot\right)  $ attains its
minimum if we also use the property of the weak lower semicontinuity of this
functional$.$ Subsequently, for any $u\in H_{0}^{^{\alpha/2}}\left(
\Omega\right)  ,$ denote
\[
F_{w}^{-}\left(  u\right)  =\min\limits_{v}F_{w}\left(  u,v\right)  .
\]
Furthermore, from $(A4)$ and $F_{w}^{-}\left(  u\right)  \leq F_{w}\left(  u,0\right)$ we obtain
\[
F_{w}^{-}\left(  u\right)  \leq\int
\nolimits_{\Omega}\left(  -\tfrac{1}{2}\left\vert (-\Delta)^{\alpha/4}u\left(
x\right)  \right\vert ^{2}+B\left\vert u\left(  x\right)  \right\vert
^{2}+\beta_{2}\left(  x\right)  u\left(  x\right)  +\gamma_{2}\left(
x\right)  \right)  dx
\]
for some constant $B$ and functions $\beta_{2}\in L^{2}\left(  \Omega\right)
$, $\gamma_{2}\in L^{1}\left(  \Omega\right)  .$\newline Using the fractional
Poincar\'{e} inequality $\left(  \ref{fPoincare}\right)  $ and the Schwartz
inequality, one gets 
%the following estimate
\begin{equation}
F_{w}^{-}\left(  u\right)  \leq\left(  -\tfrac{1}{2}+B\rho_{1}^{-\alpha
/2}\right)  \left\Vert u\right\Vert _{H_{0}^{^{\alpha/2}}}^{2}+D_{1}\left\Vert
u\right\Vert _{H_{0}^{^{\alpha/2}}}+D_{2}=p\left(  u\right)  \label{2.1}%
\end{equation}
where $D_{1}$, $D_{2}\geq0.$ It is easily seen that the functional $F_{w}^{-}$
is weakly upper semicontinous$.$ Indeed, let $u_{k}$ tend to $u_{0}$ weakly in
$H_{0}^{^{\alpha/2}}\left(  \Omega\right)  $, and let $\left\{  v_{k}\right\}
_{k\in\mathbb{N}_{0}}$ be such that $F_{w}^{-}\left(  u_{k}\right)
=F_{w}\left(  u_{k},v_{k}\right)  =\min_{v}F_{w}\left(  u_{k},v\right)  $ for
$k\in\mathbb{N}_{0}$ 
%as we have proved such a sequence $\left\{
%v_{k}\right\}  _{k\in\mathbb{N}_{0}}$ exists, 
and then
\[
\underset{k\rightarrow\infty}{\lim\sup}F_{w}^{-}\left(  u_{k}\right)
=\underset{k\rightarrow\infty}{\lim\sup}F_{w}\left(  u_{k},v_{k}\right)
\leq\underset{k\rightarrow\infty}{\lim\sup}F_{w}\left(  u_{k},v_{0}\right)
\leq F_{w}\left(  u_{0},v_{0}\right)  =F_{w}^{-}\left(  u_{0}\right)  .
\]
Since $-\frac{1}{2}+B\rho_{1}^{-\alpha/2}<0$, then for any $w\in\mathcal{W}$,
the functional $F_{w}^{-}$ attains its maximum at some point $u_{w}\in
H_{0}^{^{\alpha/2}}\left(  \Omega\right)  .$ For any point $u_{w}$ such that
\begin{equation}
F_{w}^{-}\left(  u_{w}\right)  =\max\limits_{u}F_{w}^{-}\left(  u\right)  ,
\label{2.2}%
\end{equation}
from $(A3)$ we obtain
\begin{align*}
F_{w}^{-}\left(  u_{w}\right)   &  \geq F_{w}^{-}\left(  0\right)
=\min\limits_{v}F_{w}\left(  0,v\right) \\
&  \geq\min\limits_{v}\left(  \left(  \tfrac{1}{2}-b\rho_{1}^{-\alpha
/2}\right)  \left\Vert v\right\Vert _{H_{0}^{^{\alpha/2}}}^{2}-C_{1}\left\Vert
v\right\Vert _{H_{0}^{^{\alpha/2}}}-C_{2}\right)  =\eta>-\infty
\end{align*}
where $b$, $C_{1}$, $C_{2}$, $\eta$ are some constants and $\frac{1}{2}%
-b\rho_{1}^{-\alpha/2}>0$. Note that $\eta$ does not depend on control $w$.
Moreover, it is important to notice that, for any maximizer $u_{w}$ satisfying
$(\ref{2.2})$, there exists $r_{1}>0$ such that for any $w\in\mathcal{W}$
\begin{equation}
u_{w}\in\left\{  u:F_{w}^{-}\left(  u\right)  \geq\eta\right\}  \subset
\left\{  u:p\left(  u\right)  \geq\eta\right\}  \subset B_{1}\left(
0,r_{1}\right)  \label{2.3}%
\end{equation}
where $p$ is defined in $(\ref{2.1}).$ We have checked that, for any
$w\in\mathcal{W}$, there exists a $u_{w}$
\[
F_{w}^{-}\left(  u_{w}\right)  =\max\limits_{u}F_{w}^{-}\left(  u\right)
=\max\limits_{u}\left[  \min\limits_{v}F_{w}\left(  u,v\right)  \right]  .
\]
%In a similar way 
One can show that, for any $w\in\mathcal{W}$, there
exists at least one $v_{w}\in H_{0}^{^{\alpha/2}}\left(  \Omega\right)  $ 
%such that
\begin{equation}
F_{w}^{+}\left(  v_{w}\right)  =\min\limits_{v}F_{w}^{+}\left(  v\right)
=\min\limits_{v}\left[  \max\limits_{u}F_{w}\left(  u,v\right)  \right]
\label{2.4}%
\end{equation}
where $F_{w}^{+}\left(  v\right)  =\max\limits_{u}F_{w}\left(  u,v\right)  $
and there is $r_{2}>0$ such that for $v_{w}$ satisfying $(\ref{2.4})$
\begin{equation}
v_{w}\in B_{2}\left(  0,r_{2}\right)  \label{2.4'}%
\end{equation}
Furthermore, since $v\rightarrow\max_{u}F_{w}\left(  u,v\right)  $ attains its minimum, 
there is $\lambda$ such that%
\begin{align*}
&\lambda<\min_{v}\max\limits_{u}F_{w}\left(  u,v\right)  \leq\max_{u}%
F_{w}\left(  u,0\right),\\
&\left\{  u\in H_{0}^{^{\alpha/2}}\left(  \Omega\right)  :F_{w}\left(
u,0\right)  \geq\lambda\right\}  \subset\left\{  u\in H_{0}^{^{\alpha/2}%
}\left(  \Omega\right)  :p\left(  u\right)  \geq\lambda\right\}  =A_{0}%
\end{align*}
where $p$ is defined in (\ref{2.1}). Moreover, since $A_{0}$ is relatively
compact in the weak topology of $H_{0}^{^{\alpha/2}}\left(  \Omega\right)  $
as it is a bounded subset of the reflexive space, it follows that the set
$\left\{  u\in H_{0}^{^{\alpha/2}}\left(  \Omega\right)  :F_{w}\left(
u,0\right)  \geq\lambda\right\}  $ is weakly compact. Additionally, by $(A5)$,
$F_{w}$ is concave-convex for any $w\in\mathcal{W}.$ In that way we have
demonstrated that all assertions of Ky Fan's Theorem are satisfied. Therefore,
$\max\nolimits_{u}\min\nolimits_{v}F_{w}\left(  u,v\right)  =\min
\nolimits_{v}\max\nolimits_{u}F_{w}\left(  u,v\right)  $ for any
$w\in\mathcal{W}.$ Subsequently, for any $v\in H_{0}^{^{\alpha/2}}\left(
\Omega\right)  ,$ we have
\begin{align*}
&F_{w}\left(  u_{w},v_{w}\right)     \leq\max\limits_{u}F_{w}\left(
u,v_{w}\right)  =F_{w}^{+}\left(  v_{w}\right)  =\min\limits_{v}F_{w}%
^{+}\left(  v\right)  =\min\limits_{v}\left[  \max\limits_{u}F_{w}\left(
u,v\right)  \right] \\
&  =\max\limits_{u}\left[  \min\limits_{v}F_{w}\left(
u,v\right)  \right] =\max\limits_{u}F_{w}^{-}\left(  u\right)  =F_{w}^{-}\left(  u_{w}\right)
=\min\limits_{v}F_{w}\left(  u_{w},v\right)  \leq F_{w}\left(  u_{w},v\right)
.
\end{align*}
One can verify for any $u\in H_{0}^{^{\alpha/2}}\left(
\Omega\right)  $ that 
$
F_{w}\left(  u_{w},v_{w}\right)  \geq F_{w}\left(  u,v_{w}\right)  .
$
Hence, for any $u\in H_{0}^{\alpha/2}\left(  \Omega\right)  $ and $v\in
H_{0}^{\alpha/2}\left(  \Omega\right)  ,$ the following 
$
F_{w}\left(  u,v_{w}\right)  \leq F_{w}\left(  u_{w},v_{w}\right)  \leq
F_{w}\left(  u_{w},v\right)
$
holds. Therefore, for any $w\in\mathcal{W}$, there exists at least one saddle
point of $F_{w}$ and moreover by $(\ref{2.3})$ and
$(\ref{2.4'})$, $S_{w}\subset B_{1}\left(  0,r_{1}\right)  \times B_{2}\left(
0,r_{2}\right)  .$\qed 
%This finishes the proof.
\end{proof}

%\begin{remark}
%\label{coincidence}It is well-known that the set of weak solutions of $\left(
%\ref{1.1}\right)  $ coincides with the set of saddle points of the functional
%of action defined in $\left(  \ref{1.2}\right)  $ if the functional of action
%is concave-convex. Furthermore, problem $\left(  \ref{1.1}\right)  $ has a
%unique solution if $F_{w}$ is strictly concave-strictly convex.
%\end{remark}
%-------------------------------------------------
\section{Continuous dependence\label{dependence_section}}

A natural question to ask is how $\left(  u,v\right)  $ varies as $w$ changes.
Now we look for conditions under which solutions of the variational problem
are stable. By stability here we understand the continuous dependence of
saddle points on controls. In order to state these conditions succinctly, we
introduce some notation and terminology. Let $\left\{  w_{k}\right\}
_{k\in\mathbb{N}_{0}}$ be an arbitrary sequence of elements from
$\mathcal{W}.$ Next, by $\left\{  \varphi_{k}\right\}  _{k\in\mathbb{N}_{0}}$
we denote a sequence of functionals of action such that
\begin{equation}
\varphi_{k}\left(  u,v\right)  =F_{w_{k}}\left(  u,v\right)  ,\text{ }%
k\in\mathbb{N}_{0}, \label{2.7}%
\end{equation}
where $F_{w}$ is defined in $(\ref{1.2})$ and by $S_{k}$ the set of saddle
points of $\varphi_{k}$ for $k\in\mathbb{N}_{0}$, i.e.
\begin{equation}
S_{k}=\left\{  \left(  \bar{u},\bar{v}\right)  \in\mathbb{H}_{0}^{\alpha
/2}:\varphi_{k}\left(  \bar{u},\bar{v}\right)  =\max\limits_{u}\min
\limits_{v}\varphi_{k}\left(  u,v\right)  =\min\limits_{v}\max\limits_{u}%
\varphi_{k}\left(  u,v\right)  \right\}  . \label{2.8}%
\end{equation}

In view of Theorem \ref{existence}, there exists at least one
saddle point of the functional $\varphi_{k}$, so $S_{k}$ is
nonempty and there exist $r_{1}$, $r_{2}>0$ s. t. $S_{k}\subset
B_{1}\left(  0,r_{1}\right)  \times B_{2}\left(  0,r_{2}\right).$ Before we prove the next theorem,
we recall the definition of the upper Kuratowski-Painlev\'{e} limit of the
sets $X_{k}$ in the topological space $\left(  \mathcal{H},\tau\right)  ,$
where $\left\{  X_{k}\right\}  _{k\in\mathbb{N}}$ is a sequence of subsets of
the space $\mathcal{H}$ with topology $\tau,$ cf. \cite{AubFra}. 
The upper limit of the sequence $\left\{  X_{k}\right\}  _{k\in\mathbb{N}}$ is
defined as the set of all cluster points of sequences $\left\{  x_{k}\right\}
_{k\in\mathbb{N}}$ such that $x_{k}\in X_{k}$ for $k\in\mathbb{N}.$ The upper
limit of $\left\{  X_{k}\right\}  _{k\in\mathbb{N}}$ in $\left(
\mathcal{H},\tau\right)  $\ will be denoted by $\left(  \tau\right)
\mathrm{Lim}\sup X_{k}.$
Additionally, $X_{k}$ is said to tend to $X_{0}$ in $\left(  \mathcal{H}%
,\tau\right)  $ if $\left(  \tau\right)  \mathrm{Lim}\sup X_{k}\subset X_{0}.$
In this paper $\mathcal{H=}\mathbb{H}_{0}^{\alpha/2}$ considered with the weak
topology denoted by $\left(  w\right)  $ or the strong topology denoted by
$\left(  s\right)  $, $X_{k}=S_{k}$ where $S_{k}$ is defined in $\left(
\ref{2.8}\right)  $ and $x_{k}=\left(  u_{k},v_{k}\right)  $ where $\left(
u_{k},v_{k}\right)  $ is a saddle point of $\varphi_{k}$ defined in $\left(
\ref{2.7}\right)  .$

%\subsection{Strong convergence of controls\label{subsection_strong}}

\begin{proposition}
\label{strong_weak}If conditions $(A1)-(A5)$ are satisfied and a sequence of
controls $w_{k}$ tends to $w_{0}$ in $L^{p}\left(  \Omega,\mathbb{R}%
^{m}\right)  $, then $\left(  w\right)  \mathrm{Lim}\sup S_{k}$ $\neq
\emptyset$ and $\left(  w\right)  \mathrm{Lim}\sup S_{k}\subset S_{0}$ in
$\mathbb{H}_{0}^{\alpha/2}$ where $S_{k}$ are given by $\left(  \ref{2.8}%
\right)  .$
\end{proposition}

\begin{proof}
We begin by proving that $\varphi_{k}$ converges uniformly to $\varphi_{0}$ on
$B_{1}\left(  0,r_{1}\right)  \times B_{2}\left(  0,r_{2}\right)  $ where
$B_{1}\left(  0,r_{1}\right)  $, $B_{2}\left(  0,r_{2}\right)  $ are balls
from Theorem \ref{existence} such that the set of all saddle points of
$\varphi_{k}$ denoted by $S_{k}$ is contained in $B_{1}\left(  0,r_{1}\right)
\times B_{2}\left(  0,r_{2}\right)  .$ To do this let $v\in H_{0}^{\alpha
/2}\left(  \Omega\right)  $ be an arbitrary point and suppose that, on the
contrary, the sequence $\left\{  \varphi_{k}\left(  \cdot,v\right)  \right\}$ 
does not converge to $\varphi_{0}\left(  \cdot,v\right)  $
uniformly on $B_{1}\left(  0,r_{1}\right)  .$ Then there exists a
sequence $\left\{  u_{l}\right\}  \subset B_{1}\left(  0,r_{1}\right)  $ and a
positive constant $\varepsilon$ such that
\[
\left\vert \varphi_{k}\left(  u_{l},v\right)  -\varphi_{0}\left(
u_{l},v\right)  \right\vert \geq\varepsilon\text{ for }k\in\mathbb{N}.
\]
Passing to a subsequence if necessary, one can assume that $u_{l}%
\rightharpoonup u_{0}\in B_{1}\left(  0,r_{1}\right)  $ weakly in
$H_{0}^{\alpha/2}\left(  \Omega\right)  .$ 
%It is a simple observation, 
Using
the triangle inequality, one gets that for $k\in\mathbb{N}$ 
\[\left\vert \varphi_{k}\left(  u_{l},v\right)  -\varphi_{0}\left(
u_{l},v\right)  \right\vert \le \left\vert \varphi_{k}\left(  u_{l},v\right)  -\varphi_{0}\left(
u_{0},v\right)  \right\vert  + \left\vert \varphi_{0}\left(  u_{0},v\right)  -\varphi_{0}\left(
u_{l},v\right)  \right\vert .
\]
%
%\begin{align*}
%\left\vert \varphi_{k}\left(  u_{l},v\right)  -\varphi_{0}\left(
%u_{l},v\right)  \right\vert  &  \leq\int_{\Omega}\left\vert G\left(
%x,u_{l}\left(  x\right)  ,v\left(  x\right)  ,w_{k}\left(  x\right)  \right)
%-G\left(  x,u_{l}\left(  x\right)  ,v\left(  x\right)  ,w_{0}\left(  x\right)
%\right)  \right\vert dx\\
%&  \leq\int_{\Omega}\left\vert G\left(  x,u_{l}\left(  x\right)  ,v\left(
%x\right)  ,w_{k}\left(  x\right)  \right)  -G\left(  x,u_{0}\left(  x\right)
%,v\left(  x\right)  ,w_{0}\left(  x\right)  \right)  \right\vert dx\\
%&  +\int_{\Omega}\left\vert G\left(  x,u_{l}\left(  x\right)  ,v\left(
%x\right)  ,w_{0}\left(  x\right)  \right)  -G\left(  x,u_{0}\left(  x\right)
%,v\left(  x\right)  ,w_{0}\left(  x\right)  \right)  \right\vert dx
%\end{align*}

The lower estimate by $\varepsilon$ leads to the
contradiction with the upper bound as all the above terms tend to zero. To
observe this it is enough to apply Krasnoselskii Theorem \cite[Theorem
2]{IdcRog} on the continuity of the superposition of the operators: %
$
L^{s}\left(  \Omega\right)  \times L^{p}\left(  \Omega,\mathbb{R}^{m}\right)
  \ni\left(  u,w\right)  \mapsto G\left(  \cdot,u\left(  \cdot\right)
,v\left(  \cdot\right)  ,w\left(  \cdot\right)  \right)  \in L^{1}\left(
\Omega\right)$, 
$L^{s}\left(  \Omega\right)    \ni u\mapsto G\left(  \cdot,u\left(
\cdot\right)  ,v\left(  \cdot\right)  ,w\left(  \cdot\right)  \right)  \in
L^{1}\left(  \Omega\right)$, 
since $(A2)$ holds. Next, apply the same arguments to get the uniform
convergence of the sequence $\left\{  \varphi_{k}\left(  u,\cdot\right)
\right\}$ on a ball $B_{2}\left(  0,r_{2}\right)  .$
Therefore, $\varphi_{k}\rightrightarrows\varphi_{0}$ on $B_{1}\left(
0,r_{1}\right)  \times B_{2}\left(  0,r_{2}\right)  .$ Let us denote
\[
m_{k}=\max\limits_{u}\min\limits_{v}\varphi_{k}\left(  u,v\right)
=\max\limits_{u\in B_{1}\left(  0,r_{1}\right)  }\min\limits_{v\in
B_{2}\left(  0,r_{2}\right)  }\varphi_{k}\left(  u,v\right)  \text{ for }%
k\in\mathbb{N}_{0}\text{.}%
\]
Since $\varphi_{k}\rightrightarrows\varphi_{0}$ on $B_{1}\left(
0,r_{1}\right)  \times B_{2}\left(  0,r_{2}\right)  $, for any $\varepsilon
>0,$ there exists $K_{0}$ such that%
\[
\varphi_{k}\left(  u,v\right)  \leq\varphi_{0}\left(  u,v\right)  +\varepsilon
\]
for any $\left(  u,v\right)  \in B_{1}\left(  0,r_{1}\right)  \times
B_{2}\left(  0,r_{2}\right)  $ and $k>K_{0}.$ This implies that%
\[
\min\limits_{v\in B_{2}\left(  0,r_{2}\right)  }\varphi_{k}\left(  u,v\right)
\leq\min\limits_{v\in B_{2}\left(  0,r_{2}\right)  }\varphi_{0}\left(
u,v\right)  +\varepsilon
\]
for any $u\in B_{2}\left(  0,r_{2}\right)  $ and $k>K_{0}.$ Consequently,
\[
\max\limits_{u\in B_{1}\left(  0,r_{1}\right)  }\min\limits_{v\in B_{2}\left(
0,r_{2}\right)  }\varphi_{k}\left(  u,v\right)  \leq\max\limits_{u\in
B_{1}\left(  0,r_{1}\right)  }\min\limits_{v\in B_{2}\left(  0,r_{2}\right)
}\varphi_{0}\left(  u,v\right)  +\varepsilon
\]
for $k>K_{0}.$ Thus $m_{k}-m_{0}\leq\varepsilon$ for sufficiently large $k.$
In a similar way it is possible to show that $-\varepsilon\leq m_{k}-m_{0}$
for sufficiently large $k.$ In this way we have proved that $m_{k}$ tends to
$m_{0}$ as $k\rightarrow\infty.$

Next, let $\left\{  \left(  u_{k},v_{k}\right)  \right\}$
be an arbitrary sequence of saddle points, such that $\left(  u_{k}%
,v_{k}\right)  \in S_{k}$ for $k\in\mathbb{N}.$ From Theorem \ref{existence},
for any $k\in\mathbb{N}$, the set $S_{k}$ is nonempty and there exist
$r_{1}>0$ and $r_{2}>0$ such that $S_{k}\subset B_{1}\left(  0,r_{1}\right)
\times B_{2}\left(  0,r_{2}\right)  $ for every $k$, i.e. the sequence
$\left\{  \left(  u_{k},v_{k}\right)  \right\}$ is bounded.
Moreover, the space $\mathbb{H}_{0}^{\alpha/2}$ is reflexive, which implies
that the sequence $\left\{  \left(  u_{k},v_{k}\right)  \right\}$ 
is weakly compact, therefore the set of its cluster points
with respect of weak topology of $\mathbb{H}_{0}^{\alpha/2}$ is nonempty. This
means that $\left(  w\right)  \mathrm{Lim}\sup S_{k}\neq\emptyset.$
Let $\left(  u_{0},v_{0}\right)  \in B_{1}\left(  0,r_{1}\right)  \times
B_{2}\left(  0,r_{2}\right)  $ be any cluster point of the sequence $\left\{
\left(  u_{k},v_{k}\right)  \right\} .$ Going, if necessary,
to a subsequence, we may assume that $\left\{  \left(  u_{k},v_{k}\right)
\right\} $ tends to $\left(  u_{0},v_{0}\right)  $ weakly in
$\mathbb{H}_{0}^{\alpha/2}.$ We shall show that $\left(  u_{0},v_{0}\right)
\in S_{0}.$ Suppose on the contrary that $\left(  u_{0},v_{0}\right)  $ does
not belong to $S_{0}.$ Let $\left(  \tilde{u},\tilde{v}\right)  $ be an
element of $S_{0}.$ So, we have $\varphi_{0}\left(  u_{0},v_{0}\right)
\neq\varphi_{0}\left(  \tilde{u},\tilde{v}\right)  .$ First, consider the case
when $\varphi_{0}\left(  \tilde{u},\tilde{v}\right)  -\varphi_{0}\left(
u_{0},v_{0}\right)  =\lambda<0.$ In that case we have
\begin{align*}
&m_{k}-m_{0}    =\varphi_{k}\left(  u_{k},v_{k}\right)  -\varphi_{0}\left(
u_{0},v_{0}\right)  \leq\varphi_{k}\left(  u_{k},\tilde{v}\right)
-\varphi_{0}\left(  u_{0},v_{0}\right) \\
&  =\left(  \varphi_{k}\left(  u_{k},\tilde{v}\right)  -\varphi_{0}\left(
u_{k},\tilde{v}\right)  \right)  +\left(  \varphi_{0}\left(  u_{k},\tilde
{v}\right)  -\varphi_{0}\left(  \tilde{u},\tilde{v}\right)  \right) 
  +\left(  \varphi_{0}\left(  \tilde{u},\tilde{v}\right)  -\varphi_{0}\left(
u_{0},v_{0}\right)  \right)  .
\end{align*}
From uniform convergence of $\varphi_{k}$ to $\varphi_{0}$ on $B_{1}\left(
0,r_{1}\right)  \times B_{2}\left(  0,r_{2}\right)  $ and the weak upper
semicontinuity of $\varphi_{0}\left(  \cdot,v\right)  $ we have
$
\lim\limits_{k\rightarrow\infty}\left[  \varphi_{k}\left(  u_{k},\tilde
{v}\right)  -\varphi_{0}\left(  u_{k},\tilde{v}\right)  \right]     =0 $ and 
$\limsup\limits_{k\rightarrow\infty}\left[  \varphi_{0}\left(  u_{k},\tilde
{v}\right)  -\varphi_{0}\left(  \tilde{u},\tilde{v}\right)  \right]   
\leq0.$
This implies that $\limsup\nolimits_{k\rightarrow\infty}\left(  m_{k}%
-m_{0}\right)  \leq\lambda<0.$ We have thus got a contradiction with the
previously proved fact that $m_{k}\rightarrow m_{0}$ as $k\rightarrow\infty$.
Similarly, we obtain a contradiction in the case when $\lambda>0.$ Therefore,
$\left(  u_{0},v_{0}\right)  \in S_{0}$, $\left(  w\right)  \mathrm{Lim}%
\sup S_{k}\subset S_{0}$ in $\mathbb{H}_{0}^{\alpha/2},$ which ends the proof. \qed
\end{proof}

\begin{proposition}
\label{strong_strong}If conditions $(A1)-(A5)$ are satisfied and $w_{k}$ tends
to $w_{0}$ in $L^{p}\left(  \Omega,\mathbb{R}^{m}\right)  $, then $\left(
s\right)  \mathrm{Lim}\sup S_{k}$ $\neq\emptyset$ and $\left(  s\right)
\mathrm{Lim}\sup S_{k}\subset S_{0}$ in $\mathbb{H}_{0}^{\alpha/2}.$
\end{proposition}

\begin{proof}
We start with a proof of the uniform convergence of $\varphi_{k}^{\prime}$ to
$\varphi_{0}^{\prime}$ on $B_{1}\left(  0,r_{1}\right)  \times B_{2}\left(
0,r_{2}\right)  $ where as before $B_{1}\left(  0,r_{1}\right)  $,
$B_{2}\left(  0,r_{2}\right)  $ are balls from Theorem \ref{existence} such
that for all $w\in\mathcal{W},$ the set of all saddle points of $\varphi_{k}$
denoted by $S_{k}$ is a subset of $B_{1}\left(  0,r_{1}\right)  \times
B_{2}\left(  0,r_{2}\right)  .$\newline Let $v\in H_{0}^{\alpha/2}\left(
\Omega\right)  $ be an arbitrary point. First, suppose that the sequence
$\left\{\ \frac{\partial\varphi_{k}}{\partial u}\left(  \cdot,v\right)\right\}
$ does not converge to $\frac{\partial\varphi_{0}%
}{\partial u}\left(  \cdot,v\right)  $ uniformly on $B_{1}\left(
0,r_{1}\right)  .$ This means that there exists a sequence $\left\{
u_{l}\right\}  \subset B_{1}\left(  0,r_{1}\right)  $ and a positive constant
$\varepsilon$ such that
\[
\left\vert \left\langle \tfrac{\partial\varphi_{k}}{\partial u}\left(
u_{l},v\right)  -\tfrac{\partial\varphi_{0}}{\partial u}\left(  u_{l}%
,v\right)  ,g_{l}\right\rangle \right\vert \geq\varepsilon\text{ for }%
k\in\mathbb{N}%
\]
and $\left\{  g_{l}\right\}  \subset B_{1}\left(  0,r_{1}\right)  .$ Passing
to a subsequence if necessary, assume that $u_{l}\rightharpoonup u_{0}\in
B_{1}\left(  0,r_{1}\right)  .$ Clearly  
$\left\vert \left\langle \tfrac{\partial\varphi_{k}}{\partial u}\left(
u_{l},v\right)  -\tfrac{\partial\varphi_{0}}{\partial u}\left(  u_{l}%
,v\right)  ,g_{l}\right\rangle \right\vert $ can be estimated by
\[
\left\vert \left\langle \tfrac{\partial\varphi_{k}}{\partial u}\left(
u_{l},v\right)  -\tfrac{\partial\varphi_{0}}{\partial u}\left(  u_{0}%
,v\right)  ,g_{l}\right\rangle +
 \left\langle \tfrac{\partial\varphi_{0}}{\partial u}\left(
u_{0},v\right)  -\tfrac{\partial\varphi_{0}}{\partial u}\left(  u_{l}%
,v\right)  ,g_{l}\right\rangle \right\vert .
\]
%
%\begin{align*}
%\left\vert \left\langle \tfrac{\partial\varphi_{k}}{\partial u}\left(
%u_{l},v\right)  -\tfrac{\partial\varphi_{0}}{\partial u}\left(  u_{l}%
%,v\right)  ,g_{l}\right\rangle \right\vert  &  \leq\int_{\Omega}\left\vert
%\left(  G_{u}\left(  x,u_{l}\left(  x\right)  ,v\left(  x\right)
%,w_{k}\left(  x\right)  \right)  -G_{u}\left(  x,u_{l}\left(  x\right)
%,v\left(  x\right)  ,w_{0}\left(  x\right)  \right)  \right)  g_{l}\left(
%x\right)  \right\vert dx\\
%&  \leq\int_{\Omega}\left\vert G_{u}\left(  x,u_{l}\left(  x\right)  ,v\left(
%x\right)  ,w_{k}\left(  x\right)  \right)  -G_{u}\left(  x,u_{0}\left(
%x\right)  ,v\left(  x\right)  ,w_{0}\left(  x\right)  \right)  \right\vert
%\left\vert g_{l}\left(  x\right)  \right\vert dx\\
%&  +\int_{\Omega}\left\vert \left(  G_{u}\left(  x,u_{l}\left(  x\right)
%,v\left(  x\right)  ,w_{0}\left(  x\right)  \right)  -G_{u}\left(
%x,u_{0}\left(  x\right)  ,v\left(  x\right)  ,w_{0}\left(  x\right)  \right)
%\right)  \right\vert \left\vert g_{l}\left(  x\right)  \right\vert dx
%\end{align*}
The above terms tend to zero. This is an immediate
consequence of Krasnoselskii Theorem \cite[Theorem 2]{IdcRog} on the
continuity of the superposition of operators
$
L^{s}\left(  \Omega\right)  \times L^{p}\left(  \Omega,\mathbb{R}^{m}\right)
  \ni\left(  u,w\right)  \mapsto G_{u}\left(  \cdot,u\left(  \cdot\right)
,v\left(  \cdot\right)  ,w\left(  \cdot\right)  \right)  \in L^{\frac{s}{s-1}%
}\left(  \Omega\right) $ and
$L^{s}\left(  \Omega\right)     \ni u\mapsto G_{u}\left(  \cdot,u\left(
\cdot\right)  ,v\left(  \cdot\right)  ,w\left(  \cdot\right)  \right)  \in
L^{\frac{s}{s-1}}\left(  \Omega\right)$
by $(A2)$ and using the fact that the sequence $\left\{  g_{l}\right\}  $ is
bounded$.$ Next, in similar fashion, the uniform convergence of the sequence
$\left\{  \frac{\partial\varphi_{k}}{\partial u}\left(  u,\cdot\right)
\right\} $ on a ball $B_{2}\left(  0,r_{2}\right)  $ can be
easily verified. As a result, $\varphi_{k}^{\prime}\rightrightarrows
\varphi_{0}^{\prime}$ on $B_{1}\left(  0,r_{1}\right)  \times B_{2}\left(
0,r_{2}\right)  .$ Let $\left\{  \left(  u_{k},v_{k}\right)  \right\}  \subset\mathbb{H}%
_{0}^{\alpha/2}$ be a sequence such that $\left(  u_{k},v_{k}\right)  \in
S_{k}$ for $k\in\mathbb{N}$. Since, for any $k\in\mathbb{N}$, $S_{k}\subset
B_{1}\left(  0,r_{1}\right)  \times B_{2}\left(  0,r_{2}\right)  $, for some
$r_{1},r_{2}>0$ (cf. Theorem \ref{existence}), we may assume, without loss of
generality, that $\left(  u_{k},v_{k}\right)  $ converges weakly to some
$\left(  u_{0},v_{0}\right)  \in B_{1}\left(  0,r_{1}\right)  \times
B_{2}\left(  0,r_{2}\right)  $ in $\mathbb{H}_{0}^{\alpha/2}$. Our aim is now
to show that $\left(  u_{k},v_{k}\right)  \rightarrow\left(  u_{0}%
,v_{0}\right)  $ strongly in $\mathbb{H}_{0}^{\alpha/2}$. Actually, by direct
calculations we get
\begin{align*}
&  \left\langle \varphi_{0}^{\prime}\left(  u_{k},v_{k}\right)  -\varphi
_{0}^{\prime}\left(  u_{0},v_{0}\right)  ,\left(  u_{0}-u_{k},v_{k}%
-v_{0}\right)  \right\rangle   =\left\Vert u_{k}-u_{0}\right\Vert _{H_{0}^{\alpha/2}}^{2}+\left\Vert
v_{k}-v_{0}\right\Vert _{H_{0}^{\alpha/2}}^{2}\\
&  +\int\nolimits_{\Omega}\left(  G_{u}\left(  x,u_{k}\left(  x\right)
,v_{k}\left(  x\right)  ,w_{0}\left(  x\right)  \right)  -G_{u}\left(
x,u_{0}\left(  x\right)  ,v_{0}\left(  x\right)  ,w_{0}\left(  x\right)
\right)  \right)  \left(  u_{0}\left(  x\right)  -u_{k}\left(  x\right)
\right)  dx\\
&  +\int\nolimits_{\Omega}\left(  G_{v}\left(  x,u_{k}\left(  x\right)
,v_{k}\left(  x\right)  ,w_{0}\left(  x\right)  \right)  -G_{v}\left(
x,u_{0}\left(  x\right)  ,v_{0}\left(  x\right)  ,w_{0}\left(  x\right)
\right)  \right)  \left(  v_{k}\left(  x\right)  -v_{0}\left(  x\right)
\right)  dx.
\end{align*}
Since $\varphi_{k}^{\prime}\rightrightarrows\varphi_{0}^{\prime}$ on
$B_{1}\left(  0,r_{1}\right)  \times B_{2}\left(  0,r_{2}\right)  $,
$\varphi_{0}^{\prime}\left(  u_{k},v_{k}\right)  \rightarrow0$ and therefore
the left side of the above equality tends to $0.$ We shall show that the last
two integrals above tend to zero. The condition $\left(  A2\right)  $ and the
H\"{o}lder inequality lead to the estimates:%
\begin{align*}
&  \left\vert \int\nolimits_{\Omega}\left(  G_{u}\left(  x,u_{k}\left(
x\right)  ,v_{k}\left(  x\right)  ,w_{0}\left(  x\right)  \right)
-G_{u}\left(  x,u_{0}\left(  x\right)  ,v_{0}\left(  x\right)  ,w_{0}\left(
x\right)  \right)  \right)  (u_{0}\left(  x\right)  -u_{k}\left(  x\right)
)dx\right\vert \\
&  \leq\left(  \int\nolimits_{\Omega}\left\vert G_{u}\left(  x,u_{k}\left(
x\right)  ,v_{k}\left(  x\right)  ,w_{0}\left(  x\right)  \right)
-G_{u}\left(  x,u_{0}\left(  x\right)  ,v_{0}\left(  x\right)  ,w_{0}\left(
x\right)  \right)  \right\vert ^{\frac{s}{s-1}}dx\right)  ^{\frac{s-1}{s}%
} \Vert u_0-u_k \Vert_{L^s}\\
%\left(  \int\nolimits_{\Omega}\left\vert u_{0}\left(  x\right)  -u_{k}\left(
%x\right)  \right\vert ^{s}dx\right)  ^{\frac{1}{s}}%
%\end{align*}
%and
%\begin{align*}
&  \left\vert \int\nolimits_{\Omega}\left(  G_{v}\left(  x,u_{k}\left(
x\right)  ,v_{k}\left(  x\right)  ,w_{0}\left(  x\right)  \right)
-G_{v}\left(  x,u_{0}\left(  x\right)  ,v_{0}\left(  x\right)  ,w_{0}\left(
x\right)  \right)  \right)  \left(  v_{k}\left(  x\right)  -v_{0}\left(
x\right)  \right)  dx\right\vert \\
&  \leq\left(  \int\nolimits_{\Omega}\left\vert G_{v}\left(  x,u_{k}\left(
x\right)  ,v_{k}\left(  x\right)  ,w_{0}\left(  x\right)  \right)
-G_{v}\left(  x,u_{0}\left(  x\right)  ,v_{0}\left(  x\right)  ,w_{0}\left(
x\right)  \right)  \right\vert ^{\frac{s}{s-1}}dx\right)  ^{\frac{s-1}{s}%
}
%\left(  \int\nolimits_{\Omega}\left\vert v_{k}\left(  x\right)  -v_{0}\left(
%x\right)  \right\vert ^{s}dx\right)  ^{\frac{1}{s}}.
\Vert v_0-v_k \Vert_{L^s}.
\end{align*}
Since $H_{0}^{\alpha/2}\left(  \Omega\right)  $ is compactly embedded into
$L^{s}\left(  \Omega\right)  $ for $s\in\left(  1,2_{\alpha}^{\ast}\right)  $
if $n>2$ and since both first integrals in the above estimates are bounded, it
follows that $\left(  u_{k},v_{k}\right)  \rightarrow\left(  u_{0}%
,v_{0}\right)  \in S_{0}$ in the strong topology of $\mathbb{H}_{0}^{\alpha
/2}$, i.e. $\left(  s\right)  \mathrm{Lim}\sup S_{k}$ $\neq\emptyset.$
Obviously, $\left(  s\right)  \mathrm{Lim}\sup S_{k}\subset S_{0}$ in
$\mathbb{H}_{0}^{\alpha/2},$ which is a direct consequence of $\left(
w\right)  \mathrm{Lim}\sup S_{k}\subset S_{0}$ in $\mathbb{H}_{0}^{\alpha/2}$
as proved in Proposition \ref{weak_strong} and the inclusion $\left(
s\right)  \mathrm{Lim}\sup S_{k}\subset\left(  w\right)  \mathrm{Lim}\sup
S_{k}.$ This concludes the proof.\qed
\end{proof}

%\begin{remark}
%In other words, from Preposition \ref{strong_strong} it follows that the
%set-valued mapping
%\[
%L^{p}\left(  \Omega,\mathbb{R}^{m}\right)  \ni w_{k}\mapsto S_{k}%
%\subset\mathbb{H}_{0}^{\alpha/2}%
%\]
%is well-defined and upper semicontinuous with respect to the strong topology
%of $L^{p}\left(  \Omega,\mathbb{R}^{m}\right)  $ and the strong topology of
%$\mathbb{H}_{0}^{\alpha/2}.$ If additionally each $S_{k}$ is a singleton,
%i.e., $S_{k}=\left\{  \left(  u_{k},v_{k}\right)  \right\}  $ then $\left(
%u_{k},v_{k}\right)  \rightarrow\left(  u_{0},v_{0}\right)  $ provided
%$w_{k}\rightarrow w_{0}$ in $L^{p}\left(  \Omega,\mathbb{R}^{m}\right)  .$
%\end{remark}

%\subsection{Weak convergence of controls\label{subsection_weak}}

To achieve stronger results which are useful in optimization theory, it is
necessary to weaken the notion of the convergence of controls. As a side
effect we should therefore narrow the class of equations under considerations.
Namely, in this section, we shall assume that the integrand $G$ is linear with
respect to control $w$, i.e. the function $G$ will take the form%
\begin{equation}
G\left(  x,u,v,w\right)  =G^{1}\left(  x,u,v\right)  +G^{2}\left(
x,u,v\right)  w \label{41}%
\end{equation}
where $G^{1}:\Omega\times\mathbb{R}^{2}\rightarrow\mathbb{R}$, $G^{2}%
:\Omega\times\mathbb{R}^{2}\rightarrow\mathbb{R}^{m}$, $w\in\mathbb{R}^{m}.$

Obviously, 
%in this case the boundary value problem $\left(  \ref{1.1}\right)
%$ takes the form%
%\begin{equation}
%\left\{
%\begin{array}
%[c]{l}%
%-(-\Delta)^{\alpha/2}u\left(  x\right)  +G_{u}^{1}\left(  x,u\left(  x\right)
%,v\left(  x\right)  \right)  +G_{u}^{2}\left(  x,u\left(  x\right)  ,v\left(
%x\right)  \right)  w\left(  x\right)  =0\\
%(-\Delta)^{\alpha/2}v\left(  x\right)  +G_{v}^{1}\left(  x,u\left(  x\right)
%,v\left(  x\right)  \right)  +G_{v}^{2}\left(  x,u\left(  x\right)  ,v\left(
%x\right)  \right)  w\left(  x\right)  =0\\
%u\left(  x\right)  =0,\text{ }v\left(  x\right)  =0
%\end{array}
%\right.
%\begin{array}
%[c]{l}%
%\text{in }\Omega\\
%\text{in }\Omega\\
%\text{on }\partial\Omega
%\end{array}
%\label{42}%
%\end{equation}
%and 
the functional of action $F_{w}\left(  u,v\right)$ now assumes the form%
\[
\int\nolimits_{\Omega}\left(  \tfrac{1}{2}\left\vert
(-\Delta)^{\alpha/4}v\left(  x\right)  \right\vert ^{2}-\tfrac{1}{2}\left\vert
(-\Delta)^{\alpha/4}u\left(  x\right)  \right\vert ^{2}+
\sum_{i=1}^2 G^{i}\left(
x,u\left(  x\right)  ,v\left(  x\right)  \right) \left(w\left(  x\right)\right)^{i-1}  \right) dx
%  +G^{2}\left(  x,u\left(
%x\right)  ,v\left(  x\right)  \right)  w\left(  x\right)  \right) 
\]
where $u\in H_{0}^{\alpha/2}\left(  \Omega\right)  $, $v\in H_{0}^{\alpha
/2}\left(  \Omega\right)  $, $w\in L^{p}\left(  \Omega,\mathbb{R}^{m}\right)
\,$with $p>1.$ Assume

\begin{enumerate}
\item[(A1')] the functions $G^{1},G_{u}^{1},G_{v}^{1},G^{2},G_{u}^{2}%
,G_{v}^{2}$ are measurable with respect to $x$ for any $\left(  u,v\right)
\in\mathbb{R}^{2}$ and continuous with respect to $\left(  u,v\right)  $ for
a.e. $x\in\Omega$;

\item[(A2')] for $p\in\left(  1,\infty\right)  ,$ there exists 
$c>0$ such that for any $z\in\{u,v\}$%
\begin{align*}
\left\vert G_{z}^{1}\left(  x,u,v\right)  \right\vert   & \leq c\left(
1+\left\vert u\right\vert ^{s-1}+\left\vert v\right\vert ^{s-1}\right) \\
%\left\vert G_{v}^{1}\left(  x,u,v\right)  \right\vert  &  \leq c\left(
%1+\left\vert u\right\vert ^{s-1}+\left\vert v\right\vert ^{s-1}\right) \\
\left\vert G_{z}^{2}\left(  x,u,v\right)  \right\vert  &  \leq c\left(
1+\left\vert u\right\vert ^{s-1-\frac{s}{p}}+\left\vert u\right\vert
^{s-1-\frac{s}{p}}\right) 
%\\ \left\vert G_{v}^{2}\left(  x,u,v\right)  \right\vert  &  \leq c\left(
%1+\left\vert u\right\vert ^{s-1-\frac{s}{p}}+\left\vert u\right\vert
%^{s-1-\frac{s}{p}}\right)
\end{align*}
for $x\in\Omega$ a.e., $u\in\mathbb{R}$, $v\in\mathbb{R}$ and $s\in\left(
1+\frac{1}{p-1},2_{\alpha}^{\ast}\right)  $ where $2_{\alpha}^{\ast}=\frac
{2n}{n-\alpha}>2$ and $p>\frac{2n}{n+\alpha}$; for $p=\infty,$ there
exist $c>0$ such that for any $z\in \{u,v\}$, $i\in\{1,2\}$%
\begin{equation*}
\left\vert G_{z}^{i}\left(  x,u,v\right)  \right\vert    \leq c\left(
1+\left\vert u\right\vert ^{s-1}+\left\vert v\right\vert ^{s-1}\right) \\
%\left\vert G_{v}^{1}\left(  x,u,v\right)  \right\vert  &  \leq c\left(
%1+\left\vert u\right\vert ^{s-1}+\left\vert v\right\vert ^{s-1}\right) \\
%\left\vert G_{z}^{2}\left(  x,u,v\right)  \right\vert  &  \leq c\left(
%1+\left\vert u\right\vert ^{s-1}+\left\vert v\right\vert ^{s-1}\right) 
%\\ \left\vert G_{v}^{2}\left(  x,u,v\right)  \right\vert  &  \leq c\left(
%1+\left\vert u\right\vert ^{s-1}+\left\vert v\right\vert ^{s-1}\right)
\end{equation*}
for $x\in\Omega$ a.e., $u\in\mathbb{R}$, $v\in\mathbb{R}$ and $s\in\left(
1,2_{\alpha}^{\ast}\right)  .$
\end{enumerate}

Obviously, assumptions $\left(  A1^{\prime}\right)  $, $\left(  A2^{\prime
}\right)  $ imply that the function $G$ satisfies $\left(  A1\right)  $ and
$\left(  A2\right)  .$ Moreover, we shall suppose that the function $G$ given
by $\left(  \ref{41}\right)  $ meets conditions $(A3),\left(  A4\right)
,\left(  A5\right)  .$ For this more specific form of the problem, the claim
of the theorem on the existence and the continuous dependence can be
strengthened. To draw the same conclusion this time, it suffices to assume
only the weak convergence of controls. Let $\left\{  w_{k}\right\}  _{k\in\mathbb{N}}$ 
be a sequence of controls. We shall prove:

\begin{proposition}
\label{weak_strong}Suppose that the function $G$ is of the form $\left(
\ref{41}\right)  $ and satisfies conditions $\left(  A1^{\prime}\right)
,\left(  A2^{\prime}\right)  ,\left(  A3\right)  ,$ $\left(  A4\right)
,\left(  A5\right)  .$ Moreover, the sequence of controls $w_{k}$ converges to
$w_{0}$ in the weak topology of $L^{p}\left(  \Omega,\mathbb{R}^{m}\right)  $
for $p\in\left(  2n/\left(n+\alpha \right),\infty\right)  .$ Then $\left(  s\right)
\mathrm{Lim}\sup S_{k}$ $\neq\emptyset$ and $\left(  s\right)  \mathrm{Lim}%
\sup S_{k}\subset S_{0}$ in $\mathbb{H}_{0}^{\alpha/2}.$
\end{proposition}

\begin{proof}
The proof is similar in spirit to that of Propositions \ref{strong_weak} and
\ref{strong_strong}. Although this proof runs along similar lines, there is
need of some subtle adjustments required to fit the arguments to new
framework. In fact, to prove that $\left(  w\right)  \mathrm{Lim}\sup S_{k}$
$\neq\emptyset$ and $\left(  w\right)  \mathrm{Lim}\sup S_{k}\subset S_{0}$ in
$\mathbb{H}_{0}^{\alpha/2}$ we proceed along the same lines as in the proof of
Proposition \ref{strong_weak}. The only thing to check now is the uniform
convergence of $\varphi_{k}$ to $\varphi_{0}$ on $B_{1}\left(  0,r_{1}\right)
\times B_{2}\left(  0,r_{2}\right)  .$\newline Let $v\in H_{0}^{\alpha
/2}\left(  \Omega\right)  $ be an arbitrary point. Suppose, to derive a
contradiction, that the sequence $\left\{  \varphi_{k}\left(  \cdot,v\right)
\right\}$ does not converge to $\varphi_{0}\left(
\cdot,v\right)  $ uniformly on $B_{1}\left(  0,r_{1}\right)  .$ This means
that there exist a sequence $\left\{  u_{l}\right\}  \subset B_{1}\left(
0,r_{1}\right)  $ and a positive constant $\varepsilon$ such that
$
\left\vert \varphi_{k}\left(  u_{l},v\right)  -\varphi_{0}\left(
u_{l},v\right)  \right\vert \geq\varepsilon\text{ for }k\in\mathbb{N}.
$
Passing, if necessary, to a subsequence let us assume that $u_{l}
\rightharpoonup u_{0}\in B_{1}\left(  0,r_{1}\right).$ By direct
calculations, we get
\begin{align*}
&\left\vert \varphi_{k}\left(  u_{l},v\right)  -\varphi_{0}\left(
u_{l},v\right)  \right\vert  \leq \int_{\Omega}\left\vert G^{2}\left(  x,u_{0}\left(  x\right)  ,v\left(
x\right)  \right)  \left(  w_{k}\left(  x\right)  -w_{0}\left(  x\right)
\right)  \right\vert dx \\
%\leq\int_{\Omega}\left\vert G^{2}\left(
%x,u_{l}\left(  x\right)  ,v\left(  x\right)  \right)  \left(  w_{k}\left(
%x\right)  -w_{0}\left(  x\right)  \right)  \right\vert dx\\
%&  \leq\int_{\Omega}\left\vert \left(  G^{2}\left(  x,u_{l}\left(  x\right)
%,v\left(  x\right)  \right)  -G^{2}\left(  x,u_{0}\left(  x\right)  ,v\left(
%x\right)  \right)  \right)  \left(  w_{k}\left(  x\right)  -w_{0}\left(
%x\right)  \right)  \right\vert dx\\
%&  +\int_{\Omega}\left\vert G^{2}\left(  x,u_{0}\left(  x\right)  ,v\left(
%x\right)  \right)  \left(  w_{k}\left(  x\right)  -w_{0}\left(  x\right)
%\right)  \right\vert dx\\
& +\left(  \int_{\Omega}\left\vert G^{2}\left(  x,u_{l}\left(  x\right)
,v\left(  x\right)  \right)  -G^{2}\left(  x,u_{0}\left(  x\right)  ,v\left(
x\right)  \right)  \right\vert ^{\frac{p}{p-1}}dx\right)  ^{\frac{p-1}{p}
}\Vert w_k-w_0\Vert_{L^p}
%\left(  \int_{\Omega}\left\vert w_{k}\left(  x\right)  -w_{0}\left(
%x\right)  \right\vert ^{p}dx\right)  ^{\frac{1}{p}}.
\end{align*}
for $k\in\mathbb{N}.$ Now we end up with a contradiction with 
supposition since the above integrals tend to zero. To observe this
convergence one can invoke \cite[Theorem 2]{IdcRog} to get the continuity of
the mapping
$
L^{s}\left(  \Omega\right)  \times L^{s}\left(  \Omega\right)  \ni\left(
u,v\right)  \mapsto G^{2}\left(  \cdot,u\left(  \cdot\right)  ,v\left(
\cdot\right)  \right)  \in L^{\frac{p}{p-1}}\left(  \Omega\right)
$
and use the assumption $\left(  A2^{\prime}\right)  $ together with the weak
convergence of controls in $L^{p}\left(  \Omega,\mathbb{R}^{m}\right)  .$ The
same arguments apply to the uniform convergence of the sequence $\left\{
\varphi_{k}\left(  u,\cdot\right)  \right\} $ on a ball
$B_{2}\left(  0,r_{2}\right)  .$ In that way one can demonstrate that
$\varphi_{k}\rightrightarrows\varphi_{0}$ on $B_{1}\left(  0,r_{1}\right)
\times B_{2}\left(  0,r_{2}\right)  .$ To prove that $\left(  s\right)  \mathrm{Lim}\sup S_{k}$ $\neq\emptyset$ and
$\left(  s\right)  \mathrm{Lim}\sup S_{k}\subset S_{0}$ in $\mathbb{H}%
_{0}^{\alpha/2}$ we proceed in the exactly same way as in the proof of
Proposition \ref{strong_strong}. We need to show that $\varphi_{k}^{\prime}$
converges uniformly to $\varphi_{0}^{\prime}$ on $B_{1}\left(  0,r_{1}\right)
\times B_{2}\left(  0,r_{2}\right)  .$ Let $v\in H_{0}^{\alpha/2}\left(  \Omega\right)  $ be an arbitrary point. In a
contradiction with the claim, suppose that the sequence $\left\{
\frac{\partial\varphi_{k}}{\partial u}\left(  \cdot,v\right)  \right\}$ 
does not converge to $\frac{\partial\varphi_{0}}{\partial
u}\left(  \cdot,v\right)  $ uniformly on $B_{1}\left(  0,r_{1}\right)  .$ This
means again that there exist a sequence $\left\{  u_{l}\right\}  \subset
B_{1}\left(  0,r_{1}\right)  $ and a positive constant $\varepsilon$ such
that
\[
\left\vert \left\langle \tfrac{\partial\varphi_{k}}{\partial u}\left(
u_{l},v\right)  -\tfrac{\partial\varphi_{0}}{\partial u}\left(  u_{l}%
,v\right)  ,g_{l}\right\rangle \right\vert \geq\varepsilon\text{ for }%
k\in\mathbb{N}%
\]
and $\left\{  g_{l}\right\}  \subset B_{1}\left(  0,r_{1}\right)  .$ Passing
to a subsequence one can assume that $u_{l}\rightharpoonup u_{0}\in
B_{1}\left(  0,r_{1}\right)  .$ It can be easily verified that for $k\in\mathbb{N}$%
\begin{align*}
&  \left\vert \left\langle \tfrac{\partial\varphi_{k}}{\partial u}\left(
u_{l},v\right)  -\tfrac{\partial\varphi_{0}}{\partial u}\left(  u_{l}%
,v\right)  ,g_{l}\right\rangle \right\vert \leq \int_{\Omega}\left\vert G_{u}^{2}\left(  x,u_{0}\left(  x\right)
,v\left(  x\right)  \right)  \left(  w_{k}\left(  x\right)  -w_{0}\left(
x\right)  \right)  \right\vert \left\vert g_{l}\left(  x\right)  \right\vert
dx \\
%&  \leq\int_{\Omega}\left\vert \left(  G_{u}^{2}\left(  x,u_{l}\left(
%x\right)  ,v\left(  x\right)  \right)  w_{k}\left(  x\right)  -G_{u}%
%^{2}\left(  x,u_{l}\left(  x\right)  ,v\left(  x\right)  \right)  \right)
%w_{0}\left(  x\right)  g_{l}\left(  x\right)  \right\vert dx\\
%&  \leq\int_{\Omega}\left\vert \left(  G_{u}^{2}\left(  x,u_{l}\left(
%x\right)  ,v\left(  x\right)  \right)  -G_{u}^{2}\left(  x,u_{0}\left(
%x\right)  ,v\left(  x\right)  \right)  \right)  \left(  w_{k}\left(  x\right)
%-w_{0}\left(  x\right)  \right)  \right\vert \left\vert g_{l}\left(  x\right)
%\right\vert dx\\
%&  +\int_{\Omega}\left\vert G_{u}^{2}\left(  x,u_{0}\left(  x\right)
%,v\left(  x\right)  \right)  \left(  w_{k}\left(  x\right)  -w_{0}\left(
%x\right)  \right)  \right\vert \left\vert g_{l}\left(  x\right)  \right\vert
%dx\\
&  \leq\left(  \int_{\Omega}\left\vert G_{u}^{2}\left(  x,u_{l}\left(
x\right)  ,v\left(  x\right)  \right)  -G_{u}^{2}\left(  x,u_{0}\left(
x\right)  ,v\left(  x\right)  \right)  \right\vert ^{\frac{ps}{p\left(
s-1\right)  -s}}\right)  ^{\frac{p\left(  s-1\right)  -s}{ps}}\left\Vert
w_{k}-w_{0}\right\Vert _{L^{p}}\left\Vert g_{l}\right\Vert _{L^{s}}. 
\end{align*}
The only thing to check is the convergence to zero of
the above integrals. The assumption $\left(  A2^{\prime}\right)  $,
boundedness of the sequences $\left\{  g_{l}\right\}  $, $\left\{  \left\Vert
w_{k}\right\Vert _{L^{p}}\right\}  $ as well as continuity of the operators
$L^{s}\left(  \Omega\right)     \ni u\mapsto G_{u}^{2}\left(  \cdot,u\left(
\cdot\right)  ,v\left(  \cdot\right)  \right)  \in L^{\frac{ps}{p\left(
s-1\right)  -s}}\left(  \Omega,\mathbb{R}^{m}\right)$, 
$L^{s}\left(  \Omega\right)  \times L^{p}\left(  \Omega,\mathbb{R}^{m}\right)
  \ni\left(  u,w\right)  \mapsto G_{u}^{2}\left(  \cdot,u\left(
\cdot\right)  ,v\left(  \cdot\right)  \right)  w\left(  \cdot\right)  \in
L^{\frac{s}{s-1}}\left(  \Omega\right)$,
make it possible to draw the desired conclusion. Likewise, one can show the
uniform convergence of the sequence $\left\{  \frac{\partial\varphi_{k}%
}{\partial u}\left(  u,\cdot\right)  \right\}$ on a ball
$B_{2}\left(  0,r_{2}\right)  .$ Therefore, $\varphi_{k}^{\prime
}\rightrightarrows\varphi_{0}^{\prime}$ on $B_{1}\left(  0,r_{1}\right)
\times B_{2}\left(  0,r_{2}\right)  .$ The rest follows as the
proofs of Propositions \ref{strong_weak} and \ref{strong_strong}.\qed
\end{proof}

\begin{proposition}
\label{weak*_storng}Assume that $G$ is of the form $\left(  \ref{41}\right)  $
and satisfies conditions $\left(  A1^{\prime}\right),$ $\left(  A2^{\prime
}\right)  ,\left(  A3\right)  ,$ $\left(  A4\right)  ,\left(  A5\right)$ while 
 the controls $w_{k}$ tend to $w_{0}$ in the weak $\ast$
topology of $L^{\infty}\left(  \Omega,\mathbb{R}^{m}\right)  .$ Then $\left(
s\right)  \mathrm{Lim}\sup S_{k}$ $\neq\emptyset$ and $\left(  s\right)
\mathrm{Lim}\sup S_{k}\subset S_{0}$ in $\mathbb{H}_{0}^{\alpha/2}.$
\end{proposition}

\begin{proof}
[Sketch of the proof]As it was pointed out in the proof of Proposition
\ref{weak_strong}, all we need is to demonstrate that $\varphi_{k}%
\rightrightarrows\varphi_{0}$ on $B_{1}\left(  0,r_{1}\right)  \times
B_{2}\left(  0,r_{2}\right)  $ and $\varphi_{k}^{\prime}\rightrightarrows
\varphi_{0}^{\prime}$ on $B_{1}\left(  0,r_{1}\right)  \times B_{2}\left(
0,r_{2}\right)  .$ Assume on the contrary. Note the following estimates
(analogously we consider the sequence $\left\{  v_{l}\right\}  $ and an
arbitrary $u)$%
\begin{equation*}%
\begin{array}
[l]{l}%
\left\vert \varphi_{k}\left(  u_{l},v\right)  -\varphi_{0}\left(
u_{l},v\right)  \right\vert \leq
\int_{\Omega}\left\vert G^{2}\left(  x,u_{0}\left(  x\right)
,v\left(  x\right)  \right)  \left(  w_{k}\left(  x\right)  -w_{0}\left(
x\right)  \right)  \right\vert dx\\
+\int_{\Omega}\left\vert \left(
G^{2}\left(  x,u_{l}\left(  x\right)  ,v\left(  x\right)  \right)
-G^{2}\left(  x,u_{0}\left(  x\right)  ,v\left(  x\right)  \right)  \right)
\left(  w_{k}\left(  x\right)  -w_{0}\left(  x\right)  \right)  \right\vert dx,\\
%\end{array}
%\end{equation}
%and
%\begin{equation}%
%\begin{array}[l]{l}%
\left\vert \left\langle \tfrac{\partial\varphi_{k}}{\partial u}\left(
u_{l},v\right)  -\tfrac{\partial\varphi_{0}}{\partial u}\left(  u_{l}%
,v\right)  ,g_{l}\right\rangle \right\vert \leq
\int_{\Omega}\left( \left\vert G_{u}^{2}\left(  x,u_{0}\left(  x\right)
,v\left(  x\right)  \right)  \left(  w_{k}\left(  x\right)  -w_{0}\left(
x\right)  \right)  \right\vert \left\vert g_{l}\left(  x\right)  \right\vert \right.\\
\left. +\left\vert
\left(  G_{u}^{2}\left(  x,u_{l}\left(  x\right)  ,v\left(  x\right)  \right)
-G_{u}^{2}\left(  x,u_{0}\left(  x\right)  ,v\left(  x\right)  \right)
\right)  \left(  w_{k}\left(  x\right)  -w_{0}\left(  x\right)  \right)
\right\vert \left\vert g_{l}\left(  x\right)  \right\vert \right)dx
\end{array}
\end{equation*}
for fixed $v\in H_{0}^{\alpha/2}\left(  \Omega\right)  $ and some sequences
$\left\{  u_{l}\right\}  \subset B_{1}\left(  0,r_{1}\right)  ,$ $\left\{
g_{l}\right\}  \subset B_{1}\left(  0,r_{1}\right)  .$ Since $\left\{
w_{k}\right\}$ tends to $w_{0}$ in the weak $\ast$ topology
of $L^{\infty}\left(  \Omega,\mathbb{R}^{m}\right)  $ and operator
$L^{s}\left(  \Omega\right)  \times L^{s}\left(  \Omega\right)     \ni\left(
u,v\right)  \mapsto Q\left(  \cdot,u\left(  \cdot\right)  ,v\left(
\cdot\right)  \right)  \in L^{1}\left(  \Omega\right)$, for $Q\in \{G^{2},G^{2}_u\}$,
%L^{s}\left(  \Omega\right)  \times L^{s}\left(  \Omega\right)   &  \ni\left(
%u,v\right)  \mapsto G_{u}^{2}\left(  \cdot,u\left(  \cdot\right)  ,v\left(
%\cdot\right)  \right)  \in L^{1}\left(  \Omega\right)  ,
is continuous, it follows that
%as $k\rightarrow\infty,$ that
%\begin{align*}
%\int_{\Omega}\left\vert G^{2}\left(  x,u_{0}\left(  x\right)  ,v\left(
%x\right)  \right)  \left(  w_{k}\left(  x\right)  -w_{0}\left(  x\right)
%\right)  \right\vert dx  &  \rightarrow0,\\
%\int_{\Omega}\left\vert G_{u}^{2}\left(  x,u_{0}\left(  x\right)  ,v\left(
%x\right)  \right)  \left(  w_{k}\left(  x\right)  -w_{0}\left(  x\right)
%\right)  \right\vert dx  &  \rightarrow0.
%\end{align*}
%Furthermore, due to the continuity of the following operators 
%$L^{s}\left(  \Omega\right)     \ni u\mapsto G_{u}^{2}\left(  \cdot,u\left(
%\cdot\right)  ,v\left(  \cdot\right)  \right)  \in L^{1}\left(  \Omega\right)$,
%$L^{s}\left(  \Omega\right)     \ni u\mapsto G^{2}\left(  \cdot,u\left(
%\cdot\right)  ,v\left(  \cdot\right)  \right)  \in L^{1}\left(  \Omega\right)$
%and boundedness of $\left\{  \left\Vert w_{k}-w_{0}\right\Vert _{L^{\infty}
%}\right\}  $ we get that 
all right side of the above integrals 
%in $\left(\ref{*1}\right)  $ and $\left(  \ref{*2}\right)  $ 
tend to zero that contradicts supposition. The rest follows the lines of the
proofs of Prop. \ref{strong_weak} and \ref{strong_strong}.\qed
\end{proof}

%\begin{remark}
%In Propositions \ref{weak_strong} and \ref{weak*_storng}, we have proved that
%the set-valued mapping
%\[
%L^{p}\left(  \Omega,\mathbb{R}^{m}\right)  \ni w_{k}\mapsto S_{k}%
%\subset\mathbb{H}_{0}^{\alpha/2}%
%\]
%is well-defined and upper semicontinuous with respect to either the weak
%topology of $L^{p}\left(  \Omega,\mathbb{R}^{m}\right)  $ for $p\in\left(
%\frac{2n}{n+\alpha},\infty\right)  $ or the weak $\ast$ topology of
%$L^{\infty}\left(  \Omega,\mathbb{R}^{m}\right)  ,$ and the strong topology of
%$\mathbb{H}_{0}^{\alpha/2}.$ If additionally each $S_{k}$ is a singleton,
%i.e., $S_{k}=\left\{  \left(  u_{k},v_{k}\right)  \right\}  ,$ then $\left(
%u_{k},v_{k}\right)  \rightarrow\left(  u_{0},v_{0}\right)  $ in $\mathbb{H}%
%_{0}^{\alpha/2}$ provided either $w_{k}\rightharpoonup w_{0}$ weakly in
%$L^{p}\left(  \Omega,\mathbb{R}^{m}\right)  $ or $w_{k}\overset{\ast
%}{\rightharpoonup}w_{0}$ weakly $\ast$ in $L^{\infty}\left(  \Omega
%,\mathbb{R}^{m}\right)  .$
%\end{remark}

%-------------------------------------------------

\section{Existence of optimal solutions\label{optimal_section}}

We now formulate the optimal control problem. It transpires that the continuous dependence results from Section
\ref{dependence_section} enable us to prove a theorem on the existence of optimal
processes to some optimal control problem. Specifically, we shall consider
control problem governed by boundary value problem $\left(  \ref{1.1}\right)  $
with the cost functional%
\begin{equation}
J\left(  u,v,w\right)  =\int_{\Omega}\theta\left(  x,u\left(  x\right)
,(-\Delta)^{\alpha/4}u\left(  x\right)  ,v\left(  x\right)  ,(-\Delta
)^{\alpha/4}v\left(  x\right)  ,w\left(  x\right)  \right)  dx \label{5.1}%
\end{equation}
where $\theta:\Omega\times\mathbb{R}^{4+m}\rightarrow\mathbb{R}$ is a given
function. Here $\left(  u,v\right)  \in\mathbb{H}_{0}^{\alpha/2}$ is the
trajectory and $w\in\mathcal{W}=\left\{  w\in L^{p}\left(  \Omega,\mathbb{R}^{m}\right)  :w\left(
x\right)  \in M\text{ for a.e. }x\in\Omega\right\}$ with $p\in\left(  \frac{2n}{n+\alpha},\infty\right]  $ and $M$ being a compact
and convex subset of $\mathbb{R}^{m}.$ Let $\mathcal{D}$ be all admissible triples
\[
\mathcal{D=}\left\{  \left(  u,v,w\right)  \in\mathbb{H}_{0}^{\alpha/2}%
\times\mathcal{W}:\left(  u,v\right)  \text{ is a weak solution to }\left(
\ref{1.1}\right)  \text{ for }w\in\mathcal{W}\right\}  .
\]
Under assumptions of Theorem \ref{existence} the set
of all admissible triples $\mathcal{D}$ is nonempty. In this section, our aim is to find a triple $\left(
u_{w^{\ast}},v_{w^{\ast}},w^{\ast}\right)  \in\mathcal{D}$ that minimizes the cost given by the
functional $\left(  \ref{5.1}\right)  ,$ i.e. we look for a triple $\left(
u_{w^{\ast}},v_{w^{\ast}},w^{\ast}\right)  $ satisfying%
\begin{equation}
J\left(  u_{w^{\ast}},v_{w^{\ast}},w^{\ast}\right)  =\min\limits_{\left(
u,v,w\right)  \in\mathcal{D}}J\left(  u,v,w\right)  . \label{min}%
\end{equation}

On the integrand $\theta$ of the cost functional $\left(
\ref{5.1}\right)  $ we impose the following:

\begin{enumerate}
\item[(A6)] the function $\theta=\theta\left(  x,u,p,v,q,w\right)  $ is
measurable with respect to $x$ for all $\left(  u,p,v,q,w\right)
\in\mathbb{R}^{4}\times M$, continuous w. r. t. $\left(
u,p,v,q,w\right)  $ for a.e. $x\in\Omega$ and convex w. r. t. $w$ for
all $\left(  u,p,v,q\right)  \in\mathbb{R}^{4}$ and a.e. $x\in\Omega$.
Moreover, there exists a constant $c>0$ such that%
\[
\left\vert \theta\left(  x,u,p,v,q,w\right)  \right\vert \leq c\left(
1+\left\vert u\right\vert ^{s}+\left\vert p\right\vert ^{2}+\left\vert
v\right\vert ^{s}+\left\vert q\right\vert ^{2}\right)
\]
for a.e. $x\in\Omega$, all $u\in\mathbb{R}$, $p\in\mathbb{R}$, $v\in
\mathbb{R}$, $q\in\mathbb{R}$, $w\in M$ and for some $s\in(1,2_{\alpha}^{\ast
})$;

\item[(A7)] there exist a function $\eta\in L^{1}\left(  \Omega\right)  $ and
a constant $C>0$ such that
\[
\theta\left(  x,u,p,v,q,w\right)  \geq\eta\left(  x\right)  -C\left(
\left\vert u\right\vert +\left\vert p\right\vert +\left\vert v\right\vert
+\left\vert q\right\vert +\left\vert w\right\vert \right)
\]
for all $u\in\mathbb{R}$, $p\in\mathbb{R}$, $v\in\mathbb{R}$, $q\in\mathbb{R}%
$, $w\in M$ and a.e. $x\in\Omega$.
\end{enumerate}

Now we prove a theorem on the existence of optimal processes to $\left(  \ref{min}\right)  .$

\begin{theorem}
If the function $G$ of the form $\left(  \ref{41}\right)  $ satisfies $\left(
A1^{\prime}\right)  $, $\left(  A2^{\prime}\right)  $, $\left(  A3\right)  $,
$(A4)$, $\left(  A5\right)  $ and the integrand $\theta$ meets assumptions
$\left(  A6\right)  $, $\left(  A7\right)  $, then the optimal control problem
$\left(  \ref{min}\right)  $ possesses at least one optimal process $\left(
u_{w^{\ast}},v_{w^{\ast}},w^{\ast}\right)  .\label{main_th}$
\end{theorem}

\begin{proof}
From $\left(  A6\right)  $, $\left(  A7\right)  $ and classical theorems on
semicontinuity of integral functional, see 
\cite[Theorem 1.1]{Olech} or \cite[Theorem 5]{Iof}, we deduce that $J$ is
lower semicontinuous with respect to the strong topology in the space
$\mathbb{H}_{0}^{\alpha/2}$ and either the weak topology of $L^{p}\left(
\Omega,\mathbb{R}^{m}\right)  $ for $p\in\left(2n / \left(n+\alpha\right)%
,\infty\right)  $ or the weak $\ast$ topology of $L^{\infty}\left(
\Omega,\mathbb{R}^{m}\right)  $, since convergence of any sequence $\left\{
u_{k}\right\}$ in $H_{0}^{\alpha/2}\left(  \Omega\right)  $
implies the strong convergence of $\left\{  u_{k}\right\}$
in $L^{s}\left(  \Omega\right)  $ with $s\in\left(  1,2_{\alpha}^{\ast
}\right)  $ and the strong convergence of $\left\{  (-\Delta)^{\alpha/4}%
u_{k}\right\}$ in $L^{2}\left(  \Omega\right)  $ and
moreover we have the same implications for convergence of any sequence
$\left\{  v_{k}\right\}$ in $H_{0}^{\alpha/2}\left(
\Omega\right)  .$

Next, let $\left\{  \left(  u_{k},v_{k},w_{k}\right)  \right\}\subset\mathcal{D}$ 
be a minimizing sequence for $\left(  \ref{min}\right)  $, i.e.
\begin{equation}
\lim\limits_{k\rightarrow\infty}J\left(  u_{k},v_{k},w_{k}\right)
=\inf\limits_{\left(  u,v,w\right)  \in\mathcal{D}}J\left(  u,v,w\right)
=\vartheta.\label{4.4*}%
\end{equation}
Since the set $M$ is compact and convex, we see that the sequence $\left\{
w_{k}\right\}$ is compact in the weak topology of
$L^{p}\left(  \Omega,\mathbb{R}^{m}\right)  $ for $p\in\left(
2n / \left(n+\alpha\right),\infty\right)  $ or the weak $\ast$ topology of $L^{\infty
}\left(  \Omega,\mathbb{R}^{m}\right)  ,$ respectively$.$ Passing to
subsequence if necessary, one can assume that $\left\{w_{k}\right\}$ tends to some $w_{0}%
\in\mathcal{W}$ weakly in $L^{p}\left(  \Omega,\mathbb{R}^{m}\right)  $ or
$\left\{w_{k}\right\}$ tends to some $w_{0}\in\mathcal{W}$ weakly $\ast$ in $L^{\infty
}\left(  \Omega,\mathbb{R}^{m}\right)  ,$ respectively$.$ By assumption
$\left(  A5\right)  ,$ the set of the weak solutions of problem $\left(
\ref{1.1}\right)  $ coincides with the set of saddle points of the functional
$F_{w_{k}}$ on the space $\mathbb{H}_{0}^{\alpha/2}.$ 
By Propositions \ref{weak_strong} or \ref{weak*_storng},
the sequence $\left\{  \left(  u_{k},v_{k}\right)  \right\},$ 
or at least some its subsequence, tends to $\left(  u_{0},v_{0}\right)  $
in $\mathbb{H}_{0}^{\alpha/2}$ and the triple $\left(  u_{0},v_{0}%
,w_{0}\right)  $ is an admissible triple for control problem $\left(
\ref{1.1}\right)  $. Due to the lower semicontinuity of $J$, we have
\begin{equation}
J\left(  u_{0},v_{0},w_{0}\right)  \leq\liminf\limits_{k\rightarrow\infty
}J\left(  u_{k},v_{k},w_{k}\right)  \label{4.5*}%
\end{equation}
provided $\left\{\left(  u_{k},v_{k}\right) \right\} $ tends to $\left(  u_{0},v_{0}\right)
$ in $\mathbb{H}_{0}^{\alpha/2}$ and $w_{k}\rightharpoonup w_{0}$ weakly in
$L^{p}\left(  \Omega,\mathbb{R}^{m}\right)  $ or $w_{k}\overset{\ast
}{\rightharpoonup}w_{0}$ weakly $\ast$ in $L^{\infty}\left(  \Omega
,\mathbb{R}^{m}\right)  ,$ respectively. Furthermore, by $(\ref{4.4*})$ and
$(\ref{4.5*})$
\[
\vartheta\leq J\left(  u_{0},v_{0},w_{0}\right)  \leq\liminf
\limits_{k\rightarrow\infty}J\left(  u_{k},v_{k},w_{k}\right)  =\inf
\limits_{\left(  u,v,w\right)  \in\mathcal{D}}J\left(  u,v,w\right)
=\vartheta.
\]
Thus, $J\left(  u_{0},v_{0},w_{0}\right)  =\vartheta=\inf\nolimits_{\left(
u,v,w\right)  \in\mathcal{D}}J\left(  u,v,w\right)  .$ It means that the
process $\left(  u_{w^{\ast}},v_{w^{\ast}},w^{\ast}\right)  =\left(
u_{0},v_{0},w_{0}\right)  $ is optimal for $\left(  \ref{min}\right)  $.\qed
\end{proof}

\begin{example}
\label{przyklad1}Let $
\Omega=P^{3}\left(  0,\pi\right)  =\left\{  x\in\mathbb{R}^{3}:0<x_{i}%
<\pi,\text{ }i=1,2,3\right\}.$
Note that $u_{1}=\sin x_{1}\sin x_{2}\sin x_{3}$ and $\rho_{1}=3$ are
eigenfunction and eigenvalue for $-\Delta$ on $H_{0}^{1}\left(  \Omega\right)
$ since $-\Delta u_{1}=3u_{1}.$ Similarly, $\left(  -\Delta\right)
^{\alpha/2}u_{1}=3^{\alpha/2}u_{1}$ hence, $3^{\alpha/2}$ is the first
eigenvalue for $\left(  -\Delta\right)  ^{\alpha/2}$ in this case. Consider
the following linear control problem
\begin{equation}
\left\{
\begin{array}
[c]{l}%
-(-\Delta)^{\alpha/2}u\left(  x\right)  +\beta_{1}u\left(  x\right)
+w_{1}\left(  x\right)  v\left(  x\right)  +l_{1}\left(  x\right)  =0\\
(-\Delta)^{\alpha/2}v\left(  x\right)  -\beta_{2}v\left(  x\right)
+w_{2}\left(  x\right)  u\left(  x\right)  +l_{2}\left(  x\right)  =0\\
u\left(  x\right)  =0,\text{ }v\left(  x\right)  =0
\end{array}
\right.
\begin{array}
[c]{c}%
\text{in }\Omega\\
\text{in }\Omega\\
\text{on }\partial\Omega
\end{array}
\label{3.5}%
\end{equation}
with $\beta_{i}<3^{\alpha/2},$ $l_{i}\in L^{2}\left(  \Omega\right)  $, $i=1,2$, 
$
\mathcal{W}=\left\{  w\in L^{p}\left(  \Omega,\mathbb{R}^{2}\right)  :w \in\left[  0,1\right]^2  \text{ a.e. on } \Omega\right\}
$ with $p\in\left(3/\alpha,\infty\right)  .$ The functional of action
for control problem $(\ref{3.5})$ is of the form
\begin{align*}
F_{w}\left(  u,v\right)   &  =\int\nolimits_{\Omega}\left(  \tfrac{1}%
{2}\left\vert (-\Delta)^{\alpha/4}v\left(  x\right)  \right\vert ^{2}%
-\tfrac{1}{2}\left\vert (-\Delta)^{\alpha/4}u\left(  x\right)  \right\vert
^{2}+\tfrac{\beta_{1}}{2}\left\vert u\left(  x\right)  \right\vert ^{2}%
-\tfrac{\beta_{2}}{2}\left\vert v\left(  x\right)  \right\vert ^{2}\right.  \\
&  \left.  +\left(  w_{1}\left(  x\right)  +w_{2}\left(  x\right)  \right)
u\left(  x\right)  v\left(  x\right)  +l_{1}\left(  x\right)  u\left(
x\right)  +l_{2}\left(  x\right)  v\left(  x\right)  \vphantom{\int_0}\right)
dx.
\end{align*}
The cost, for $s\in\left(  1,6/\left(3-\alpha\right)\right),$  is
$J\left(  u,v,w\right)  =\int\nolimits_{\Omega}\left(  u^{s}\left(  x\right)
+v^{s}\left(  x\right)  +\left\vert w\left(  x\right)  \right\vert
^{2}\right)  dx.$

%cf. Theorem \ref{main_th}.
\end{example}

\begin{example}
Let $\Omega=P^{3}\left(  0,\pi\right)  $. The control problem now is of the form%
\begin{equation*}
\left\{
\begin{array}
[c]{l}%
-(-\Delta)^{\alpha/2}u\left(  x\right)  +bu\left(  x\right)  -s\left\vert
x\right\vert ^{2}u^{s-1}\left(  x\right)  w_{1}\left(  x\right)  -\left\vert
x\right\vert w_{2}\left(  x\right)  +v\left(  x\right)  =0\\
(-\Delta)^{\alpha/2}v\left(  x\right)  -av\left(  x\right)  +s\left\vert
x\right\vert ^{2}v^{s-1}\left(  x\right)  w_{1}\left(  x\right)  -\left\vert
x\right\vert w_{2}\left(  x\right)  +u\left(  x\right)  =0\\
u\left(  x\right)  =0,\text{ }v\left(  x\right)  =0
\end{array}
\right.
\begin{array}
[c]{l}%
\text{in }\Omega\\
\text{in }\Omega\\
\text{on }\mathbb{\partial}\Omega
\end{array}
\label{4.9}%
\end{equation*}
for $1+1/\left(p-1\right)<s<6/\left(3-\alpha\right)$ with $p>6/\left(3+\alpha\right)  $ 
or $1<s<6/\left(3-\alpha\right)$ with $p=\infty.$ Now,
the cost is given by
$
J\left(  u,v,w\right)  =\int_{\Omega}\left( u^{s}\left(  x\right)
+\left\vert (-\Delta)^{\alpha/4}u\left(  x\right)  \right\vert ^{2}%
w_{1}\left(  x\right)  +\right.\\
\left.\left\vert (-\Delta)^{\alpha/4}v\left(  x\right)
\right\vert ^{2}w_{2}\left(  x\right)  -\left\vert x\right\vert (-\Delta
)^{\alpha/4}u\left(  x\right)  +\left\vert w\left(  x\right)  \right\vert
^{2}\right)  dx
%\label{4.10*}%
$
where $a<3^{\alpha/2},$ $b<3^{\alpha/2}$ and $M=\left[  0,1\right]^2.$ 
Obviously, the functional of action in this case has the form
$
F_{w}\left(  u,v\right)  =\int\nolimits_{\Omega}  \left(\tfrac{1}%
{2}\left\vert (-\Delta)^{\alpha/4}v\left(  x\right)  \right\vert ^{2}%
-\tfrac{1}{2}\left\vert (-\Delta)^{\alpha/4}u\left(  x\right)  \right\vert
^{2}-\tfrac{a}{2}v^{2}\left(  x\right)  +\tfrac{b}{2}u^{2}\left(  x\right)+\right.\\
\left.\left\vert x\right\vert ^{2}v^{s}\left(  x\right)  w_{1}\left(  x\right)
  -\left\vert x\right\vert ^{2}u^{s}\left(  x\right)  w_{1}\left(
x\right)  -u\left(  x\right)  \left\vert x\right\vert w_{2}\left(  x\right)
-v\left(  x\right)  \left\vert x\right\vert w_{2}\left(  x\right)  +u\left(
x\right)  v\left(  x\right)\right)   dx
$.
One can check that $F_{w}$, $J$ satisfy
assumptions of Theorems \ref{existence} and \ref{main_th}.
% By Remark
%\ref{uwaga}, $F_{w}$ is strictly concave in $u$ and strictly convex in $v$.
%Thus, Theorem \ref{existence} and Remark \ref{coincidence} imply that for any
%control $w$ there exists exactly one weak solution $\left(  u_{w}%
%,v_{w}\right)  $ of control problem $\left(  \ref{4.9}\right)  $ and moreover
%from Propositions \ref{weak_strong} and \ref{weak*_storng} one can deduce that
%the weak solution continuously depends on control $w$ provided controls
%converge in the weak topology of $L^{p}\left(  \Omega,\mathbb{R}^{2}\right)  $
%with $p\in\left(  \frac{6}{3-\alpha},\infty\right)  $ or the weak $\ast$
%topology of $L^{\infty}\left(  \Omega,\mathbb{R}^{2}\right)  ,$
%respectively$.$ Moreover, from Theorem \ref{main_th}, we infer that there
%exists an optimal control $w^{\ast}$ such that the triple $\left(  u_{w^{\ast
%}},v_{w^{\ast}},w^{\ast}\right)  $ is an admissible triple of the process
%described by control problem $\left(  \ref{4.9}\right)  $ that minimizes the
%cost functional given by $\left(  \ref{4.10*}\right)  .$
\end{example}

%\section{Concluding remarks\label{remarks_section}}

%As far as we know, the question of the continuous dependence on functional
%parameters or controls of the solutions of control problem governed by
%fractional differential equations involving the Dirichlet fractional Laplacian
%has not been considered up to now. In this paper, we have established the
%existence and continuous dependence on the functional parameter of weak
%solutions corresponding to saddle critical points of the functional of action.
%Furthermore, the existence of optimal processes minimizing the cost functional
%was ascertained. The novelty of the results lies in the nonlocal structure of
%both action and cost functionals depending on the values of nonlocal Dirichlet
%fractional Laplace operator.

%\begin{acknowledgements}
%\end{acknowledgements}



\end{document}